\documentclass{amsart}

\usepackage{amsfonts}
\usepackage{amssymb}
\usepackage{tikz}
\usepackage{caption}
\usepackage{cancel}
\usepackage{bm,braket,faktor,stmaryrd}

\usepackage{graphicx}
\usepackage{bm}
\usepackage[
	pdftitle={},
	pdfauthor={},
	ocgcolorlinks,
	linkcolor=linkblue,
	citecolor=linkred,
	urlcolor=linkred]
{hyperref}
\usepackage{color}
\definecolor{linkred}{rgb}{0.75,0,0}
\definecolor{linkblue}{rgb}{0,0,0.75}
\usepackage{enumitem}

\theoremstyle{plain}
\newtheorem{theorem}{Theorem}
\newtheorem{proposition}{Proposition}[section]
\newtheorem{lemma}[proposition]{Lemma}
\newtheorem{thm}[proposition]{Theorem}

\newtheorem{cor}[proposition]{Corollary}

\newcommand{\bt}{\begin{theorem}}
\newcommand{\et}{\end{theorem}}

\theoremstyle{definition}
\newtheorem{definition}[proposition]{Definition}
\newtheorem{remark}[proposition]{Remark}
\newcommand{\beq}{\begin{equation}}
\newcommand{\eeq}{\end{equation}}
\newcommand{\bl}{\begin{lemma}}
\newcommand{\el}{\end{lemma}}

\newcommand{\cal}{\mathcal}

\newcommand{\cc}{\mathcal{C}}

\newcommand{\cl}{\mathcal{L}}
\newcommand{\ce}{\mathcal{E}}
\newcommand{\co}{\mathcal{O}}

\newcommand{\ct}{\mathcal{T}}
\newcommand{\cu}{\mathcal{U}}

\newcommand{\cf}{{\cal F}}
\newcommand{\cz}{{\cal Z}}

\newcommand{\bc}{\mathbb{C}}

\newcommand{\bk}{\mathbb{K}}
\newcommand{\bn}{\mathbb{N}}

\newcommand{\bq}{\mathbb{Q}}
\newcommand{\br}{\mathbb{R}}
\newcommand{\bz}{\mathbb{Z}}

\newcommand{\WP}{^{WP}\hspace{-.4mm}}

\newcommand{\modm}{\cal M}
\newcommand{\un}{1\!\!1}

\allowdisplaybreaks

\begin{document}
	
\title{Super Weil-Petersson measures on the moduli space of curves}
\author{Paul Norbury}
\address{School of Mathematics and Statistics, University of Melbourne, VIC 3010, Australia}
\email{\href{mailto:norbury@unimelb.edu.au}{norbury@unimelb.edu.au}}
\thanks{}
\date{\today}

\begin{abstract}
%In this paper we define finite measures on the moduli space of smooth curves with marked points.   The measures arise naturally out of the super Weil-Petersson metric defined over the moduli space of super curves and the total measures can be identified with the volume of the moduli space of super curves.   The construction of the measures is analogous to the construction of the Weil-Petersson metric, using the extra data of a spin structure.   The measures are indexed by the behaviour of the spin structure at marked points, labeled Neveu-Schwarz or Ramond.  The Neveu-Schwarz volume polynomials satisfy a recursion relation discovered by Stanford and Witten, analogous to Mirzakhani's recursion relations between Weil-Petersson volumes of moduli spaces of hyperbolic surfaces.  We prove here that the Ramond boundary behaviour produces deformations of the Neveu-Schwarz volume polynomials, satisfying a variant of the Stanford and Witten recursion relations.

The super Weil-Petersson metric defined over the moduli space of smooth super curves produces a natural measure over the moduli space of smooth curves.  The construction of the measure uses the extra data of a spin structure on each smooth curve.  When we allow marked points, the construction produces a collection of measures indexed by the behaviour of the spin structure at marked points---Neveu-Schwarz or Ramond.  In this paper we define these measures, and prove that they are finite.  Each total measure gives the super volume of the moduli space of super curves with marked points.  The Neveu-Schwarz volumes are polynomials that satisfy a recursion relation discovered by Stanford and Witten, analogous to Mirzakhani's recursion relations between Weil-Petersson volumes of moduli spaces of hyperbolic surfaces.  We prove here that the Ramond boundary behaviour produces deformations of the Neveu-Schwarz volume polynomials, satisfying a variant of the Stanford-Witten recursion relations.

\end{abstract}
\maketitle

\tableofcontents

\section{Introduction}  \label{sec:intro}

The moduli space $\modm_{g,n}$ of genus $g$ curves with $n$ marked points comes equipped with a natural symplectic form $\omega^{WP}$ known as the Weil-Petersson form.  It is the imaginary part of the natural Hermitian metric on the (co)tangent bundle over $\modm_{g,n}$ due to Weil \cite{WeiMod} defined by using the complete hyperbolic metric in a conformal class following Petersson \cite{PetUbe}.  At a point $(C,p_1,\dots,p_n)\in\modm_{g,n}$, the Hermitian metric is defined via the Petersson pairing on the vector space of quadratic differentials
\begin{equation}  \label{WPmetric}
\langle\eta,\xi\rangle:=\int_\cc\frac{\overline{\eta}\xi}{h},\qquad \eta,\xi\in H^0(C,\omega_{\cc}^{\otimes 2}(D))
\end{equation}
where $D=(p_1,\dots,p_n)$ and $h$ is the complete hyperbolic metric on $C-D$.  %The pairing \eqref{WPmetric} defines a Hermitian metric on $\modm_{g,n}$ with imaginary part a K\"ahler form $\omega^{WP}$.
The top power of its imaginary part $\omega^{WP}$ produces a finite measure on $\modm_{g,n}$ with total measure the Weil-Petersson volume, conveniently defined by
\[
V^{WP}_{g,n}:=\int_{\modm_{g,n}}\exp\omega^{WP}.
\]

Over the moduli space $\modm^{\text{spin}}_{g,n}$ of genus $g$ {\em spin} curves with $n$ marked points there is a natural finite measure, which %also uses the complete hyperbolic metric in any conformal class.  It 
pushes forward to define a family of finite measures on $\modm_{g,n}$.  The measure is constructed via a natural vector bundle
\[ \begin{array}{c}E_{g,n}\\\downarrow\\\modm^{\text{spin}}_{g,n}\end{array}
\] 
and the Weil-Petersson form $\omega^{WP}$ defined over $\modm^{\text{spin}}_{g,n}$ by pullback under the map that forgets the spin structure.  The bundle 
$E_{g,n}$ can be identified with the normal bundle of $\modm^{\text{spin}}_{g,n}\to\widehat{\modm}_{g,n}$, the moduli space of genus $g$ supercurves with $n$ marked points containing its reduced space.     The vector bundle $E_{g,n}$ is defined in Definition~\ref{obsbun} using only classical constructions, i.e. without reference to the super moduli space.  Just as the Weil-Petersson metric is naturally defined on the cotangent bundle of $\modm_{g,n}$, the dual bundle $E_{g,n}^\vee$, which has fibre given by $3/2$ differentials, in place of quadratic differentials, arises in a similar way, also via Serre duality.  %It is not important, since an Euler form of a bundle and its dual naturally agree up to a possible change in sign.  
 A natural Euler form on $E_{g,n}^\vee$ is defined, in the so-called Neveu-Schwarz case, using the following natural Hermitian metric:
\begin{equation}  \label{hermetric} 
\langle\eta,\xi\rangle:=\int_\cc\frac{\overline{\eta}\xi}{\sqrt{h}},\qquad \eta,\xi\in H^0(C,\omega_C^{\otimes3/2}(D)).
\end{equation}
The measure on $\modm^{\text{spin}}_{g,n}$ is given by the top degree part of the differential form $e(E_{g,n}^\vee)\exp(\omega^{WP})$, which turns out to define a {\em finite} measure on $\modm^{\text{spin}}_{g,n}$.  
The natural Hermitian metric \eqref{hermetric} on $E_{g,n}^\vee$ uses the complete hyperbolic metric on each curve in a similar way to the construction of the Weil-Petersson form \eqref{WPmetric}.  The more general case of \eqref{hermetric}  for Neveu-Schwarz and Ramond points allows different pole behaviour at the marked points and is given in Definition~\ref{hermetricNSR}.
The close resemblance of \eqref{WPmetric} and \eqref{hermetric}  reflects the fact that \eqref{hermetric} arises from a super-generalisation of \eqref{WPmetric}.  The hyperbolic metric is described locally on $C$ by $h=|dz|^2/(\text{Im}\hspace{.4mm} z)^2$ (locally contained in the upper half plane $\text{Im}\hspace{.4mm} z>0$).  Write $\eta=f(z)dz^2$ and $\xi=g(z)dz^2$ locally, then the integrand in \eqref{WPmetric} becomes 
\[ \overline{f(z)}g(z)(\text{Im}\hspace{.4mm} z)^2|dz|^2.
\]
A super generalisation of the hyperbolic metric in super coordinates $(z|\theta)$, where $z$ is bosonic and $\theta$ is fermionic,  replaces $(\text{Im}\hspace{.4mm} z)^2$ by
\[ (\text{Im}\hspace{.4mm} z+\tfrac12\theta\bar{\theta})^2=(\text{Im}\hspace{.4mm} z)^2+\theta\bar{\theta}(\text{Im}\hspace{.4mm} z).
\]  
After integrating out the fermionic directions $d\theta d\bar{\theta}$, this produces the local integrand for $\eta=f(z)dz^{3/2}$ and $\xi=g(z)dz^{3/2}$ given by
\[ \overline{f(z)}g(z)(\text{Im}\hspace{.4mm} z)|dz|^2.
\]
which explains the factor of $1/\sqrt{h}=\text{Im}\hspace{.4mm} z$ in \eqref{hermetric}.

%See Section~\ref{sec:eulerform} for details.    

The moduli space of spin curves decomposes into components
\[\modm_{g,n}^{\text{spin}}=\bigsqcup_{\sigma\in\{0,1\}^n}\modm_{g,\sigma}^{\text{spin}}\]
where $\sigma=(\sigma_1,...,\sigma_n)\in\{0,1\}^n$ and $\sigma_j$ describes the behaviour of the spin structure at the marked point $p_j$, denoted {\em Neveu-Schwarz} when $\sigma_j=1$ and  {\em Ramond} when $\sigma_j=0$.  The vector bundle $E_{g,n}$ restricts to a rank $2g-2+\frac12(n+|\sigma|)$ bundle $E_{g,\sigma}\to\modm^{\text{spin}}_{g,\sigma}\subset\modm_{g,n}^{\text{spin}}$, for $|\sigma|=\sum\sigma_j$, and $\deg e(E_{g,\sigma}^\vee)=4g-4+n+|\sigma|$.  

Each component inherits a finite measure with associated volume,
\begin{equation}  \label{supvol}
\widehat{V}^{WP}_{g,\sigma}:=\epsilon_{g,\sigma}\int_{\modm_{g,\sigma}^{\text{spin}}}e(E_{g,\sigma}^\vee)\exp\omega^{WP}
\end{equation}
for $\epsilon_{g,\sigma}=2^{g-1+\frac12(n+|\sigma|)}$.  The factor $\epsilon_{g,\sigma}$ is chosen to make recursion and restriction formulae more natural.  This measure was defined in \cite{NorEnu} over $\modm^{\text{spin}}_{g,\sigma}$ for $\sigma=1^n$ known as the Neveu-Schwartz component (which is the union of two connected components). The volume of the Neveu-Schwartz component was calculated via recursive formulae by Stanford and Witten \cite{SWiJTG}.  This paper considers the general problem of defining and calculating the volume of any component of $\modm^{\text{spin}}_{g,n}$.

%The super volume of a supermanifold with reduced space a symplectic manifold can be expressed via the integral of an Euler form of the normal bundle of the reduced space combined with a power of the symplectic form as in \eqref{supvol}.  
Given a vector bundle $E\to M$ defined over a smooth symplectic manifold $(M,\omega)$, the sheaf of smooth sections of the exterior algebra $\Lambda^*E^\vee$ of the dual bundle $E^\vee$ defines a smooth supermanifold $\widehat{M}$ with reduced space $M$, and $\int_Me(E^\vee)\exp(\omega)$ can be interpreted as the supervolume of $\widehat{M}$.  The exterior algebra $\Lambda^*W^\vee$ of the dual of a vector space over a field $k$ is a super-commutative generalisation of the ring $\text{Sym}^*V^\vee\cong k[x_1,...,x_n]$, for $x_i\in V^\vee$. %, a fundamental ring in algebraic geometry, usually represented in the form $k[x_1,...,x_n]\cong\text{Sym}^*V^\vee$ the dual of a dimension $n$ vector space. % which is independent of the choice of the Euler form when $M$ is closed.  
The moduli space of genus $g$ super curves with $n$ marked points $\widehat{\modm}_{g,n}$ can be defined smoothly by the sheaf of smooth sections of %the exterior algebra of the bundle 
$\Lambda^*E_{g,n}^\vee$, where $E_{g,n}$ is defined in Definition~\ref{obsbun}, with reduced space the moduli space $\modm_{g,n}^{\rm spin}$ of spin curves  with $n$ marked points.  Although the results in this paper will not rely on super constructions, the super geometry viewpoint produces useful heuristic ideas. Most significantly in the Neveu-Schwarz case it led to the work of Stanford and Witten \cite{SWiJTG} which produced recursive formulae for the super volumes analogous to those of Mirzakhani.

A powerful method pioneered by Mirzakhani \cite{MirSim} to study Weil-Petersson volumes is to consider a family of symplectic deformations $\omega(L_1,...,L_n)$ of the Weil-Petersson form for $(L_1,...,L_n)\in\br_{\geq 0}^n$, which produces corresponding 
volumes 
$V_{g,n}(L_1,...,L_n):=\int_{\modm_{g,n}}\exp\omega(L_1,...,L_n)$.     
Mirzakhani proved that $V_{g,n}(L_1,...,L_n)$ is polynomial in $L_i$ and satisfies a recursion that uniquely determines all polynomials from the initial polynomials $V_{0,3}(L_1,L_2,L_3)=1$  and $V_{1,1}(L_1)=\frac{1}{48}(L^2+4\pi^2)$. 
We use the same deformation of the Weil-Petersson form to study the measure defined in \eqref{supvol}. 

It will be useful to deform the symplectic form $\omega^{WP}$ only at Neveu-Schwarz points as follows:
\begin{equation}  \label{nsdef}
\widehat{V}_{g,n}^{(m)}(L_1,...,L_n):=\epsilon_{g,\sigma}\int_{\modm_{g,\sigma}^{\rm spin}}\hspace{-6mm}e(E^\vee_{g,\sigma})\exp\omega(L_1,...,L_n,0^m),\quad \sigma=(1^n,0^m).
\end{equation}
 A consequence of Theorem~\ref{main} below, in the $(g,n)=(0,1)$ case, and for general $(g,n)$ using an algebro-geometric construction of the volume proven in Corollary~\ref{volint}, is the property that $\widehat{V}_{g,n}^{(m)}(L_1,...,L_n)$ is a polynomial in $(L_1,...,L_n)$ of degree $2g-2+m$.

For a given genus $g$, allow $n$ Neveu-Schwarz marked points and any number of Ramond points.  Deform $\omega^{WP}$ only at the Neveu-Schwarz points and collect these volumes into the following generating function over all possible numbers of Ramond points, which we refer to as a gas of Ramond points.
\begin{equation}  \label{supvolpol}
\widehat{V}_{g,n}\WP(s,L_1,...,L_n):=\sum_{m=0}^\infty\frac{s^m}{m!}\widehat{V}_{g,n}^{(m)}(L_1,...,L_n)
%\widehat{V}^{WP}_{g,n}(s,L_1,...,L_n):=\sum_{m=0}^\infty\frac{s^m}{m!}\int_{\modm_{g,(1^n,0^m)}^{\rm spin}}\hspace{-6mm}e(E_{g,(1^n,0^m)})\exp[\omega(L_1,...,L_n,0^m)].
\end{equation}
The Ramond points are essentially unordered due to the factor $1/m!$, and the only non-zero terms occur when $m$ is even due to a spin parity condition.  
We will denote the special case of $(g,n)=(0,1)$ by the {\em disk} function which gives the genus zero volumes with exactly one Neveu-Schwarz marked point and a gas of Ramond  points:
\[
\widehat{V}_{0,1}\WP(s,L):=\sum_{m=0}^\infty\frac{s^m}{m!}\widehat{V}_{0,1}^{(m)}(L).
%\widehat{V}_{0,1}(s,L):=\sum_{m=0}^\infty\frac{s^m}{m!}\int_{\modm_{0,(1,0^m)}^{\rm spin}}\hspace{-6mm}e(E_{g,(1,0^m)})\exp[\omega(L,0^m)].
\]
%which specialises at $L=0$ to $\widehat{V}_{0,1}(s,0)=\sum_{m=0}^\infty\frac{s^m}{m!}\widehat{V}^{WP}_{0,(1,0^m)}$.
Define the kernel
\[ D(x,y,z):=\frac{1}{4\pi}\left(\frac{1}{\cosh\frac{x-y-z}{4}}-\frac{1}{\cosh\frac{x+y+z}{4}}\right)
\]
which arises in \cite{NorEnu,SWiJTG} via algebraic geometry, respectively supergeometry.
\begin{theorem}  \label{main}
The disk function $\widehat{V}^{WP}_{0,1}(s,L)$ is uniquely determined by the recursion
\[ \widehat{V}_{0,1}\WP(s,L)=\frac{s^2}{2!}+\frac{1}{2L}\int_0^\infty\int_0^\infty\hspace{-2mm} xyD(L,x,y)\widehat{V}_{0,1}\WP(s,x)\widehat{V}_{0,1}\WP(s,y)dxdy.
\]
%and the initial value $\widehat{V}^{WP}_{0,(1,0,0)}=1$.
\end{theorem}
Theorem~\ref{main} defines a recursion in increasing powers of $s$.  We obtain 
\[ \widehat{V}_{0,1}\WP(s,L)=\frac{s^2}{2!}+\left(6\pi^2+\frac12L^2\right)\frac{s^4}{4!}+\left(330\pi^4 + 30\pi^2L^2 + \frac38L^4\right)\frac{s^6}{6!}+...
\]
and in particular, the $L=0$ terms produce the volumes $\widehat{V}^{WP}_{0,(1,0^m)}$. 

The disk recursion generalises to allow any $\widehat{V}_{g,n}\WP(s,\vec{L})$. 
%\[ R(x,y,z):=\frac{1}{8\pi}\left(\frac{1}{\cosh\frac{x+y-z}{4}}-\frac{1}{\cosh\frac{x+y+z}{4}}+\frac{1}{\cosh\frac{x-y-z}{4}}-\frac{1}{\cosh\frac{x-y+z}{4}}\right)\]
%which also arises in \cite{NorEnu,SWiJTG}. % and satisfies $D(x,y,z)=R(x,y,z)+R(x,z,y)$.  
For any $n>0$, write $\vec{L}=(L_1,...,L_n)\in\br^n$ and $L_I:=\vec{L}|_I\in\br^I$ its restriction to $I\subset\{1,...,n\}$.  Define a second kernel $R(x,y,z):=\tfrac12(D(x+y,z,0)+D(x-y,z,0))$.
%For $k\leq n$, define $\un_k:=(\overbrace{1,..,1}^k,0,..,0)\in\{0,1\}^n$ and $\vec{L}_k=(L_1,..,L_k,0,..,0)\in\br_{\geq 0}^n$  
\begin{theorem}   \label{conj}
The following recursion  uniquely determines $\widehat{V}_{g,n}\WP(s,\vec{L})$ up to $O(s^4)$.
\begin{align} \label{reconj}
L_1&\widehat{V}_{g,n}\WP(s,\vec{L})=\tfrac12\int_0^\infty\hspace{-2mm}\int_0^\infty\hspace{-2mm} xyD(L_1,x,y)P_{g,n+1}(x,y,L_K)dxdy\\
&+\sum_{j=2}^n\int_0^\infty \hspace{-2mm}xR(L_1,L_j,x)\widehat{V}_{g,n-1}\WP(s,x,L_{K\backslash\{j\}})dx+\delta_{1n}(\tfrac{s^2\delta_{0g}}{2}+\tfrac{\delta_{1g}}{8})L_1+O(s^4)
\nonumber 
\end{align}
for $K=(2,...,n)$ and
\[P_{g,n+1}(x,y,L_K)=\widehat{V}_{g-1,n+1}\WP(s,x,y,L_K)+\hspace{-3mm}\mathop{\sum_{g_1+g_2=g}}_{I \sqcup J = K}\hspace{-2mm}\widehat{V}_{g_1,|I|+1}\WP(s,x,L_I)\widehat{V}_{g_2,|J|+1}\WP(s,y,L_J).
\]
\end{theorem}
Theorem~\ref{conj} leads to the recursive calculation of $\widehat{V}_{g,n}^{(2)}(L_1,...,L_n)$.
The special cases of $(g,n)=(0,1)$ and $(g,n)=(1,1)$ produce the $\delta_{1n}$ term in \eqref{reconj} which encodes the initial conditions
$\widehat{V}_{0,1}^{(2)}(L)=1$ and $\widehat{V}_{1,1}^{(0)}(L)=\frac18$.  We expect that one can remove the $O(s^4)$ in \eqref{reconj} so that the recursion holds exactly, as it does in the disk case.  
The recursion resembles Mirzakhani's recursion \cite{MirSim} although it differs due to the $(0,1)$ and $(0,2)$ terms on the right hand side.  In particular this leads to the appearance of the term $\widehat{V}_{g,n}\WP(s,\vec{L})$ on both the left and right hand sides.  When $n=0$, the volume in \eqref{supvolpol} defines series $F_g(s)$ for $g\geq0$.  We can calculate these series from $\widehat{V}_{g,1}\WP(s,L)$ via the recursion \eqref{reconj} together with the relation 
\[
\left(2g-2+s\frac{\partial}{\partial s}\right)F_g(s)=\widehat{V}_{g,1}\WP(s,2\pi i)
\]
which is stated more generally in \eqref{dilaton}.

The $s=0$ case of Theorem~\ref{conj} is proven in \cite{NorEnu} via algebro-geometric techniques and in \cite{SWiJTG} via heuristic supergeometry techniques.  What is striking is that the same kernels appear in \eqref{reconj} for the $s\neq 0$ case, and in particular that the kernels are $s$-independent.

Theorems~\ref{main} and \ref{conj} are stated in terms of differential geometric structures.  Their proofs, given in Sections~\ref{stablecurves} and \ref{virastr}, use algebraic geometry, specifically properties of intersection numbers over the compactification of the moduli space.  The algebro-geometric calculations are aided by algebraic properties of partial differential Virasoro operators that naturally act on partition functions storing the intersection numbers---see Section~\ref{virastr}.  Translation of the partition functions, given by Theorem~\ref{th:transl}, simplifies the algebro-geometric calculations, and commutator properties of the operators uncover a tractable algebro-geometric problem to be solved.    Despite the success of the algebro-geometric techniques, differential geometry, via heuristic ideas from supergeometry, actually predicts the interesting structure of $\widehat{V}_{g,n}\WP(s,\vec{L})$ shown by the recursion \eqref{reconj} and differential geometry should provide the proofs.  Essentially algebraic geometry has the power to calculate and prove recursive structure, but it does not know what to prove, instead relying on guidance from (super) differential geometry.   This is analogous to Mirzakhani's work \cite{MirSim} where she used differential geometry to deduce calculations of intersection numbers in algebraic geometry.  We discuss this further in Section~\ref{supergeom} which describes an approach to a proof using super differential geometry.

{\em Acknowledgements.} I am grateful for discussions with Alexander Alexandrov, Alessandro Giacchetto and Edward Witten.

\section{Moduli space of spin curves}

A twisted, pointed curve $(\cc,p_1,\dots,p_n)$ with group $\bz_2$ is a compact orbifold such that each marked point $p_i$ has isotropy group $\bz_2$ and all other points have trivial isotropy group.  It has a canonical sheaf $\omega_{\cc}$ which defines an orbifold line bundle over $\cc$.  The moduli space of twisted, pointed curves  with group $\bz_2$ is denoted by $\modm^{(2)}_{g,n}$.  A twisted curve comes equipped with a map which forgets the orbifold structure $\rho:\cc\to C$ where $C$ is a pointed curve known as the coarse curve of $\cc$, often also denoted by $C=|\cc|$.  We abuse notation and write $D=(p_1,\dots,p_n)\subset C$ to denote the image in the coarse curve of the points $D\subset\cc$ with non-trivial isotropy.  A spin curve is a twisted curve together with a square root of its log-canonical sheaf.  Its moduli space is defined in  \cite{FJRWit} as follows:
\begin{definition}  \label{def:modspin}
The moduli space of spin curves is defined by
\[\modm_{g,n}^{\text{spin}}=\{(\cc,\theta,p_1,...,p_n,\phi)\mid \phi:\theta^2\stackrel{\cong}{\longrightarrow}\omega_{\cc}^{\text{log}}\}
\]
where $\cc$ is a twisted curve and $\theta$ is a line bundle over $\cc$.
\end{definition}
There is a natural map
\begin{equation} \label{forgets}
p:\modm_{g,n}^{\text{spin}}\to\modm_{g,n}
\end{equation}
obtained by forgetting the spin structure and mapping to the underlying coarse curve via $\rho$, i.e. $p$ is a composition of the map induced by $\rho$ together with the  $2^{2g}$ to 1 map to the moduli space of twisted curves $\modm_{g,n}^{\text{spin}}\to\modm^{(2)}_{g,n}$.  There are $2^{2g+n-1}$ choices of $(\theta,\phi)$ for each twisted, pointed curve  $(\cc,p_1,...,p_n)\in\modm^{(2)}_{g,n}$, and after fixing representation data at each $p_i$, described below, there are $2^{2g}$ different spin structures.

A spin structure on a twisted, pointed curve $(\cc,p_1,\dots,p_n)$ defines a line bundle $\theta\to\cc$ that is a square root of the log-canonical bundle $\theta^2\cong\omega_\cc^{\text{log}}$.  Both $\theta$ and $\omega_\cc^{\text{log}}$ are line bundles, also known as orbifold  line bundles, defined over $\cc$.  An orbifold  line bundle has a well-defined degree which may be a half-integer since the points with non-trivial isotropy can contribute $1/2$ to the degree.  In particular, 
\[\deg\omega_{\cc}^{\text{log}}=2g-2+n,\qquad\deg\theta=g-1+\frac12n.\] 
The degree of $\theta^{\vee}$ is negative, $\deg\theta^{\vee}=1-g-\frac12n<0$, hence  $\theta^{\vee}$ possesses no holomorphic sections and $h^0(\cc,\theta^{\vee})=0$.   The index $h^0(\cc,\theta^{\vee})-h^1(\cc,\theta^{\vee})$ is constant over any family, thus $H^1(\cc,\theta^{\vee})$ has constant dimension and defines a vector bundle $E_{g,n}$.  More precisely, denote by $\ce$ the universal spin structure defined over the universal curve  $\cu_{g,n}^{\text{spin}}\stackrel{\pi}{\longrightarrow}\modm_{g,n}^{\text{spin}}$.  
\begin{definition}  \label{obsbun}
Define the bundle $E_{g,n}:=-R\pi_*\ce^\vee\hspace{-1mm}\to\modm_{g,n}^{\text{spin}}$ with fibre $H^1(\cc,\theta^{\vee})$.   
\end{definition}    

The behaviour of a spin structure at a marked point $p$ is one of two types, known as Neveu-Schwartz and Ramond which we now define.
The orbifold line bundles $\omega_{\cc}^{\text{log}}$ and $\theta$ over $\cc$ are locally equivariant bundles over the local charts.  On each fibre over a point $p$ with non-trivial isotropy the equivariant trivialisation defines a representation acting by multiplication by $(-1)^\sigma$ for $\sigma\in\{0,1\}$.   The representation associated to $\omega_{\cc}^{\text{log}}$ at $p_i$ is trivial since $\omega_{\cc}^{\text{log}}$ is generated locally by $dz/z$ and the local $\bz_2$ action $z\mapsto -z$ induces 
\[\frac{dz}{z}\stackrel{z\mapsto-z}{\longrightarrow}\frac{dz}{z}.\]  
Note that the representation associated to $\omega_{\cc}$ at $p_i$ is non-trivial, since $dz\mapsto-dz$, so in particular $\omega_{\cc}$ cannot have a square root, which justifies considering square roots of the log canonical bundle.
The representations associated to $\theta$ at each $p_i$ define a vector $\sigma\in\{0,1\}^n$.%, where $0$, respectively $1$, in $\{0,1\}$ corresponds to the unique trival, respectively non-trivial, representation $\bz_2\to\bz_2$.   
A marked point $p_i$ is known as a {\em Neveu-Schwarz} point when the associated representation is non-trivial, and a {\em Ramond} point otherwise.  The representations at marked points define a decomposition of the moduli space into connected components 
\begin{equation}  \label{modspin}
\modm_{g,n}^{\text{spin}}=\bigsqcup_{\sigma\in\{0,1\}^n}\modm_{g,\sigma}^{\text{spin}}.
\end{equation}
The parameters $\sigma\in\{0,1\}^n$ will also label the natural measures to be constructed on $\modm_{g,n}$ as pushforwards of a measure on $\modm_{g,n}^{\text{spin}}$ restricted to each component.  The restriction of the forgetful map which forgets the spin structure and the orbifold structure to a component is denoted by $p^\sigma:\modm_{g,\sigma}^{\text{spin}}\to\modm_{g,n}$.

For any sheaf $\cf$ over $\cc$, define the pushforward sheaf $|\cf|:=\rho_*\cf$ over the coarse curve $C=|\cc|$.  The sheaf of local sections $\co_\cc(L)$ of any line bundle $L$ on $\cc$ pushes forward to a sheaf $|L|:=\rho_*\co_\cc(L)$ on $C$ which can be identified with the local sections of $L$ invariant under the $\bz_2$ action.  The pullback of the pushforward is the bundle 
$\rho^*(|\theta|)=\theta\otimes\bigotimes_{i\in I}\co(-\sigma_ip_i)$,
since locally invariant sections must vanish when the representation is non-trivial.  In particular, $\deg|\theta|=\deg \theta-\tfrac12|\sigma|$ where $\deg\theta=g-1+\frac12n$ is independent of the representations $\sigma$.  Let $U\subset\cc$ be an open set with local coordinate $z$ that contains a single marked point $p_i$ satisfying $z(p_i)=0$, and suppose that $p_i$ is Neveu-Schwarz.  Then $\rho^*(|\theta|)\otimes\rho^*(|\theta|)$
%Under the map $U\to|U|$ that forgets the orbifold structure at $p_i$, 
 lies in the image of the natural map $\omega_\cc(U)\stackrel{\cdot z}{\longrightarrow}\omega_\cc(p_i)(U)$ since local sections of the pushforward sheaf and its square vanish.  In particular, this shows that one can forget a Neveu-Schwarz point $p_{n+1}$, for $\sigma'=(\sigma,1)\in\{0,1\}^{n+1}$ :
\begin{equation}  \label{forgetp}
\pi:\modm_{g,\sigma'}^{\text{spin}}\to\modm_{g,\sigma}^{\text{spin}}
\end{equation}
where  $\pi$ forgets the orbifold structure at $p_{n+1}$ and, as explained above, the square of the spin structure takes its values in the canonical sheaf around $p_{n+1}$ in place of the log-canonical sheaf.  The map $\pi$ and the map $p$ defined in \eqref{forgets} are both forgetful maps, forgetting different structures on the curve.  We also see that % $[\rho^*(|\theta|)]^2\to\omega_\cc(R)$, and the pushforward of a spin structure is a square root of a twisted canonical  sheaf:
\[ |\theta|^2\cong\omega_C(R)
\]
where $R\subset D$ consists of the Ramond points, which agrees with the convention in \cite{ChiTow}.

The rank of the bundle $E_{g,n}$ depends on the component $\modm_{g,\sigma}^{\rm spin}$ of the moduli space.  It can be calculated using the Riemann-Roch theorem applied to the pushforward of $\theta$ to the underlying coarse curve $C$, since $H^i(\cc,\theta^{\vee})=H^i(C,|\theta^{\vee}|)$.  Apply Riemann-Roch to the coarse curve $C$ to get:
\[h^0(C,|\theta^{\vee}|)-h^1(C,|\theta^{\vee}|)=1-g+\deg|\theta^{\vee}|=2-2g-\tfrac12(n+|\sigma|).\]
and $h^0(C,|\theta^{\vee}|)=h^0(\cc,\theta^{\vee})=0$, thus the rank of $E_{g,n}$ restricted to $\modm_{g,\sigma}^{\rm spin}$ is
\begin{equation}  \label{rank}
\text{rank\ }E_{g,n}|_{\modm_{g,\sigma}^{\rm spin}}=2g-2+\tfrac12(n+|\sigma|).
\end{equation}
A consequence of \eqref{rank} is that the number of Ramond points  on any curve is even.

\subsection{Euler form over the moduli space}   \label{sec:eulerform}

We construct a canonical Euler form $e(E_{g,n}^\vee)$ associated to the bundle $E_{g,n}\to\modm_{g,n}^{\text{spin}}$ using the Chern connection of a natural Hermitian metric.    The canonical Hermitian metric on $E_{g,n}^\vee$ is defined similarly to the Hermitian metric on the cotangent bundle of $\modm_{g,n}$, which produces the Weil-Petersson metric on $\modm_{g,n}$.  Both constructions use the natural hyperbolic metric associated to each curve of the moduli space.   For a twisted, pointed curve $(\cc,D)$, where $D=(p_1,\dots,p_n)$, via Serre duality 
\[ H^1(\cc,\theta^\vee) \cong H^0(\cc,\omega_{\cc}\otimes\theta)^\vee.\]
Elements of $H^0(\cc,\omega_{\cc}\otimes\theta)$ give the analogue of meromorphic quadratic differentials used to define the Weil-Petersson metric.  
\begin{definition}  \label{hermetricNSR}
Define a Hermitian metric on the vector space $H^0(\cc,\omega_{\cc}\otimes\theta)$ by
\begin{equation}  \label{eq:hermetricNSR}
\langle\eta,\xi\rangle:=\int_\cc\frac{\overline{\eta}\xi}{\sqrt{h}},\qquad \eta,\xi\in H^0(\cc,\omega_{\cc}\otimes\theta)
\end{equation}
where $h$ is the complete hyperbolic metric on $\cc-D$ in its conformal class. 
\end{definition}

\begin{lemma}   \label{welldef}
The Hermitian metric \eqref{eq:hermetricNSR} is well-defined.
\end{lemma}
\begin{proof}We need to show that the integrand in \eqref{eq:hermetricNSR} defines an integrable two-form.
Locally, the transition functions, equivalently the sheaf restriction maps, of $\overline{\eta}\xi$ are given by $|z'(w)|^3$ where $z$ and $w$ are local coordinates on $\cc$.  The transition functions for $\sqrt{h}$ are $|z'(w)|$ so the transition functions for $\overline{\eta}\xi/\sqrt{h}$
are $|z'(w)|^2$ proving that the quotient is a 2-form which can be integrated over $\cc$.  

To prove that \eqref{eq:hermetricNSR} is convergent, consider the local situation.  Near any marked point $p_i$, with respect to a local coordinate $z$ satisfying $z(p_i)=0$, the hyperbolic metric is given by 
\[h=\frac{|dz|^2}{|z|^2(\log|z|)^2}=\frac{|dx|^2}{|x|^2(\log|x|)^2},\quad x=z^2
\] 
which uses the local formula for a cuspidal hyperbolic metric on the coarse curve with local coordinate  $x$.
Since $\omega_{\cc}$ consists of holomorphic differentials and $\theta^2\cong\omega^{\text{log}}_{\cc}$ then locally 
\[ \overline{\eta}\xi=\frac{F(z)}{|z|}|dz|^3,\quad\eta,\xi\in H^0(\cc,\omega_{\cc}\otimes\theta)
\]
for $F(z)$ a smooth function, where the square of a local section of $\theta$ contributes a pole of order 1 in $z$ so that the Hermitian product of two sections produces the pole of order 1 in $|z|$, and local sections of $\omega_{\cc}$ are smooth at $z=0$.  Hence
\[\int_{|z|<\epsilon}\frac{\overline{\eta}\xi}{\sqrt{h}}=\int_{|z|<\epsilon}F(z)|\log|z|||dz|^2\sim M\int_0^{2\pi}\int_0^\epsilon r\log(r)drd\theta<\infty
\]
due to the bounded integrand.  We conclude that \eqref{eq:hermetricNSR} is well-defined.
\end{proof}

The pushforward of elements of $H^0(\cc,\omega_{\cc}\otimes\theta)$
from $\cc$ to the coarse curve $C=|\cc|$ has a closer comparison to the meromorphic quadratic differentials with simple poles used to define the Weil-Petersson metric.  There is a map $|\omega_\cc\otimes\theta| \to\omega_C^{3/2}(D)$, hence the pushforward sections take values in $H^0(C,\omega_C^{3/2}(D))$ which resembles $H^0(\cc,\omega_{\cc}^{\otimes 2}(D))$.    For $\eta,\xi\in H^0(C,\omega_C^{3/2}(D))$, near a cusp the $3/2$ differentials are given by $\eta=\frac{f(x)dx^{3/2}}{x}$ and $\xi=\frac{g(x)dx^{3/2}}{x}$ where $f(x)$ and $g(x)$ are holomorphic at $x=0$.  Hence the local contribution to the metric  \eqref{eq:hermetricNSR} is at worst of order
\begin{equation}  \label{locest}
\int_{|z|<\epsilon}\frac{|\log|x||}{|x|}|dx|^2=\int_0^\epsilon|\log r| drd\theta<\infty
\end{equation}
which reduces to the local contribution in the proof of Lemma~\ref{welldef} via the subsitution $x=z^2$ and again we see the convergence of \eqref{eq:hermetricNSR} .

\subsubsection{Chern form}
A holomorphic bundle $E\to M$ equipped with a Hermitian metric is naturally a real oriented bundle of even rank with a Riemannian metric.  The Hermitian metric induces the Chern connection, a unique natural Hermitian connection compatible with both the holomorphic structure and the Hermitian metric, which is a metric connection with respect to the underlying Riemannian metric on $E$.  Thus the Hermitian metric defined in \eqref{eq:hermetricNSR} defines the Chern connection on $E_{g,n}^\vee\to\modm_{g,n}^{\text{spin}}$.

Any real oriented bundle $E\to M$ of rank $N$ equipped with a Riemannian metric $\langle\cdot,\cdot\rangle$ and a metric connection $A$ defines  an Euler form 
\[e(E):=\left(\frac{1}{4\pi}\right)^N\text{pf}(F_A)\in\Omega^N(M)\] 
where $\text{pf}(F_A)$ is the Pfaffian ofhe curvature of the connection $F_A\in\Omega^2(M,\text{End}(E))$, \cite{OsbRep}.   
The Bianchi identity $\nabla^AF_A=0$ implies $e(E)$ is closed, hence it represents a cohomology class on $M$.  When $M$ is compact, the cohomology class of the Euler form is independent of the choice of metric and connection, and represents the {\em Euler class} of $E$.   

Apply this construction to $E_{g,n}^\vee$ equipped with its natural Chern connection to produce a canonical Euler form $e(E_{g,n}^\vee)$.  This also gives the Chern-Weil construction of the top Chern form $e(E_{g,n}^\vee)=c_{N/2}(E_{g,n}^\vee)$.

\subsection{Measures on the moduli space of curves}
We use $e(E_{g,n}^\vee)$ to define a measure on $\modm_{g,n}^{\rm spin}$ which pushes forward to define a measure on $\modm_{g,n}$.  In this paper a measure is given by a top degree differential form.  Although differential forms naturally pull back rather than push forward, the nowhere vanishing differential $(\omega^{WP})^D$ on $\modm_{g,n}$, for $D=3g-3+n$, allows one to replace any top degree differential form $\eta$ by a function $f_\eta$ defined via $\eta=f_\eta p^*(\omega^{WP})^D$.  Then $p_*\eta:=\bar{f}_\eta(\omega^{WP})^D$ for $\bar{f}(m):=\sum_{m'\in p^{-1}(m)} f(m')$ the natural pushforward of a function.   % Its restriction to each component $\modm_{g,\sigma}^{\rm spin}$ defines different pushforward measures as we now describe.
The Euler form $e(E_{g,n}^\vee)$ restricted to $\modm_{g,\sigma}^{\rm spin}$ is a differential form of degree $4g-4+n+|\sigma|$ by \eqref{rank},
\[ e(E_{g,n}^\vee)\in\Omega^N(\modm_{g,\sigma}^{\rm spin}),\qquad N=4g-4+n+|\sigma|.
\]
Combined with the Weil-Petersson form, it defines a measure $e(E_{g,n}^\vee)\exp(\omega^{WP})$ on $\modm_{g,\sigma}^{\rm spin}$ that pushes forward to a measure on $\modm_{g,n}$,
\[\mu_\sigma:=p^\sigma_*\big\{e(E_{g,n}^\vee)\exp(\omega^{WP})\big\}. \]

Via the forgetful map $\pi^{(m)}:\modm_{g,n+m}\to\modm_{g,n}$, define the formal sum of $S_n$ symmetric measures 
on $\modm_{g,n}$:
\[ \mu(s)=\sum_{m=0}^\infty\frac{s^m}{m!}\pi^{(m)}_*\mu_{(1^n,0^m)}
\]
where the measure $\mu_m$ is defined on $\modm_{g,n+m}$.
Then $\mu(s)$ has total measure 
\[
\widehat{V}^{WP}_{g,n}(s,0,...,0)=\int_{\modm_{g,n}}\mu(s).
%\widehat{V}^{WP}_{g,n}(s,L_1,...,L_n):=\sum_{m=0}^\infty\frac{s^m}{m!}\int_{\modm_{g,(1^n,0^m)}^{\rm spin}}\hspace{-6mm}e(E_{g,(1^n,0^m)})\exp[\omega(L_1,...,L_n,0^m)].
\]
which coincides with \eqref{supvolpol} evaluated at $(L_1,...,L_n)=(0,...,0)$.
\section{Moduli space of stable curves}   \label{stablecurves}
The measures $\mu_\sigma$ are defined over the moduli space of smooth curves which is non-compact.  A powerful tool to prove general properties of these measures, particularly properties of the total measures, is via a compactification of the moduli space.  This turns the total measure, given by integration over $\modm_{g,n}$, into an algebraic topological, or cohomological, invariant and and intersection number in algebraic geometry.  The compactifications used here are the moduli space of stable curves $\overline{\modm}_{g,n}$ and a corresponding compactification of the moduli space of spin curves $\overline{\modm}_{g,n}^{\text{spin}}$.  To obtain a relationship of the measure to cohomological invariants of the compactification and to calculate, we need the following.
\begin{enumerate}[label=(\roman*)]
\item Extend $E_{g,n}\to\modm^{\text{spin}}_{g,n}$ to a bundle $\bar{E}_{g,n}\to\overline{\modm}_{g,n}^{\text{spin}}$.\label{extb}
\item Prove the existence of an extension of $e(E_{g,n}^\vee)$ to $\overline{\modm}_{g,n}^{\text{spin}}$.\label{exte}
\item Organise the algebraic structure of intersection numbers via a cohomological field theory and an associated partition function.\label{coh}
\end{enumerate}
The extension \ref{extb} is proven in \cite{NorNew}, and \ref{exte} and \ref{coh} are proven in \cite{NorEnu}.  The papers \cite{NorEnu,NorNew} emphasise the Neveu-Schwartz component, where $\sigma=1^n$, so we will recall the constructions here for both Neveu-Schwartz and Ramond points.  The construction in \ref{exte} naturally relates the extension of $e(E_{g,n}^\vee)$ to the Euler (or Chern) {\em class} of $\bar{E}_{g,n}$ enabling us to deduce that the volumes are given via cohomological calculations, or intersection numbers, over $\overline{\modm}_{g,n}^{\text{spin}}$.

\subsection{Moduli space of stable spin curves}

A stable twisted curve, with group $\bz_2$, is a 1-dimensional orbifold, or stack, and underlying coarse curve $\cc\to|\cc|$ given by a stable curve.  The marked points $p_i\in\cc$ and nodes of $\cc$ have isotropy group $\bz_2$, while all other points have trivial isotropy group.   A line bundle $L$ over $\cc$ is a locally equivariant bundle over the local charts, such that at each nodal point there is an equivariant isomorphism of fibres.  The representations of $\bz_2$ on $L|_p$ agree on each local irreducible component at a node due to the equivariant isomorphisms at nodes.

\begin{definition}  \label{def:modspincomp}
The moduli space of stable spin curves is defined by
\[\overline{\modm}_{g,n}^{\text{spin}}=\{(\cc,\theta,p_1,...,p_n,\phi)\mid \phi:\theta^2\stackrel{\cong}{\longrightarrow}\omega_{\cc}^{\text{log}}\}
\]
where $\theta$ is a  line bundle over a stable, twisted curve $\cc$ with group $\bz_2$, each nodal point and marked point $p_i$ has isotropy group $\bz_2$, and all other points have trivial isotropy group.
\end{definition}
This gives a natural compactification of the moduli space of twisted, stable, spin curves $\modm_{g,n}^{\text{spin}}$.  We denote the extension of the forgetful map $p:\modm_{g,n}^{\text{spin}}\to\modm_{g,n}$, which forgets the spin and orbifold structure, with the same name: 
\[ p:\overline{\modm}_{g,n}^{\text{spin}}\to\overline{\modm}_{g,n}.
\]
The moduli space of stable spin curves decomposes as in \eqref{modspin} into components $\overline{\modm}_{g,\sigma}^{\text{spin}}\subset\overline{\modm}_{g,n}^{\text{spin}}$ for each $\sigma\in\{0,1\}^n$.

Denote by $\ce$ the universal spin structure on the universal curve $\overline{\cu}_{g,n}^{\text{spin}}\stackrel{\pi}{\longrightarrow}\overline{\modm}_{g,n}^{\text{spin}}$ using the same notation as its restriction to the moduli space of smooth spin curves.  
\begin{definition}  
Define the bundle $\bar{E}_{g,n}:=-R\pi_*\ce^\vee\hspace{-1mm}\to\overline{\modm}_{g,n}^{\text{spin}}$ with fibre $H^1(\cc,\theta^{\vee})$.   
\end{definition}    
This defines a bundle, i.e. the dimension of  $H^1(\cc,\theta^{\vee})$ is constant, due to the vanishing $H^0(\cc,\theta^{\vee})=0$.  The proof that $H^0(\cc,\theta^{\vee})=0$ when $\cc$ is smooth used the fact that $\deg\theta^{\vee}<0$.  When $\cc$ is nodal, the proof is subtler.  For any irreducible component $X\subset\cc$, $\omega_{\cc}^{\text{log}}|_X=\omega_X^{\text{log}}$, hence 
\[\deg\theta^{\vee}|_{X}=-\frac12\deg\omega_{\cc}^{\text{log}}|_X=-\frac12\deg\omega_X^{\text{log}}<0\] 
since the stability of $\cc$ implies $\deg\omega_X^{\text{log}}>0$ for any irreducible component.  We conclude that $H^0(\cc,\theta^{\vee})=0$ since it is trivial on each irreducible component. 

By construction, 
\[\bar{E}_{g,n}|_{\modm_{g,n}^{\text{spin}}}=E_{g,n}
\]
hence $\bar{E}_{g,n}$ defines an extension of $E_{g,n}$ to $\overline{\modm}_{g,n}^{\text{spin}}$.
The rank of $\bar{E}_{g,n}$ restricted to the component $\overline{\modm}_{g,\sigma}^{\text{spin}}$ is $2g-2+\tfrac12(n+|\sigma|)$ which is derived in \eqref{rank}.

\subsubsection{Extension of Euler class}

The following theorem states that the Euler form $e(E_{g,n}^\vee)$ extends to the compactification $\overline{\modm}_{g,n}^{\text{spin}}$.  On each component $\modm_{g,\sigma}^{\text{spin}}\subset\overline{\modm}_{g,\sigma}^{\text{spin}}$, and  defines a cohomology class in $H^{4g-4+n+|\sigma|}(\overline{\modm}_{g,\sigma}^{\text{spin}},\br)$.  This extension is a consequence of the fact that the Hermitian metric that defines $e(E_{g,n}^\vee)$ extends smoothly from $E_{g,n}^\vee$ to its extension $\bar{E}_{g,n}^\vee$ to $\overline{\modm}_{g,\sigma}^{\text{spin}}$.  This enables us to conclude that the cohomology class defined by the extension of $e(E_{g,n}^\vee)$ coincides with the Euler class of $\bar{E}_{g,n}^\vee$.

\begin{thm}[\cite{NorEnu}]  \label{eulerfc}
The extension of the Euler form $e(E_{g,n}^\vee)$ to $\overline{\modm}_{g,\sigma}^{\text{spin}}$ defines a cohomology class which coincides with the Euler class $e(\bar{E}_{g,n}^\vee)\in H^*(\overline{\modm}_{g,\sigma}^{\text{spin}},\br)$.
\end{thm}
\begin{proof}
We prove that the Hermitian metric \eqref{eq:hermetricNSR} on $E_{g,n}^\vee$ extends to a Hermitian metric on the bundle $\bar{E}_{g,n}^\vee\to\overline{\modm}_{g,n}^{\text{spin}}$ by analysing the behaviour of the poles of the $3/2$ differentials representing fibres of $E_{g,n}^\vee$.
Let $(\cc,D,\theta,\phi)\in\overline{\modm}_{g,n}^{\text{spin}}$ such that $\cc$ is a nodal curve.  The pullback of $\theta$ to the normalisation of $\cc$ is an orbifold bundle on each component.  In particular, points in the fibre of $\bar{E}_{g,n}^\vee$ given by elements of $H^0(\cc,\omega_{\cc}\otimes\theta)$ have the same simple pole behaviour at nodes and at marked points.  The pole at a node is present if the behaviour at the node is Neveu-Schwarz and removable if the behaviour at the node is Ramond.  Thus the estimate \eqref{locest} applies also at nodes to prove that the Hermitian metric on $H^0(\cc,\omega_{\cc}\otimes\theta)$ is well-defined when $\cc$ is nodal.  
The conclusion is that the Hermitian metric on $E_{g,n}^\vee$ extends to a Hermitian metric on $\bar{E}^\vee_{g,n}$.  Furthermore, it extends to a smooth Hermitian metric on $\bar{E}^\vee_{g,n}$ because the hyperbolic metric $h$ varies smoothly outside of nodes and has a canonical form around nodes, and the Hermitian metric is defined via an integral over $1/\sqrt{h}$ times smooth sections.  %Stated in another way, the Hermitian metric on $\bar{E}_{g,n}$ can be written locally as a $(2g-2+n)\times(2g-2+n)$ matrix $H$ defined via $H_{ij}=\langle s_i,s_j\rangle$ where $\{s_i\}$ is a local holomorphic basis for $\bar{E}_{g,n}$.  Derivatives of the entries $H_{ij}$ are given by derivatives of the integral \eqref{eq:hermetricNSR} where around nodes the metric and sections have a canonical local form.

We conclude that the Euler form $e(E_{g,n}^\vee)$, constructed from the curvature of the natural metric connection $A$, which is determined uniquely from the Hermitian metric and the holomorphic structure on $E_{g,n}^\vee$, extends to $\overline{\modm}_{g,n}^{\text{spin}}$.  The Euler class of $\bar{E}^\vee_{g,n}$ is determined by a choice of any metric connection on $\bar{E}^\vee_{g,n}$, in particular the Chern connection of the extension of the Hermitian metric on $\bar{E}^\vee_{g,n}$.  Thus the cohomology class defined by the extension of $e(E_{g,n}^\vee)$ coincides with the Euler class $e(\bar{E}_{g,n})\in H^*(\overline{\modm}_{g,n}^{\text{spin}},\br)$.
\end{proof}
\begin{remark}   
The Weil-Petersson Hermitian metric \eqref{WPmetric} on the (co)tangent bundle of $\modm_{g,n}$ does not extend to $\overline{\modm}_{g,n}$ since it blows up as a cusp forms in a family of hyperbolic metrics.  This is a consequence of the fact that a meromorphic quadratic differential $\eta\in H^0(|\cc|,\omega_{|\cc|}^{\otimes 2}(D))$ has simple poles at marked points and double poles at nodes. This contrasts with the behaviour of the Hermitian metric defined on $E_{g,n}^\vee$ which does extend to $\overline{\modm}_{g,n}^{\text{spin}}$ since elements of $H^0(\cc,\omega_{\cc}\otimes\theta)$ are holomorphic at marked points {\em and} nodes.  Equivalently, for a closer comparison, their pushforward to the coarse curve has simple poles at marked points and nodes.  %See \cite{NorEnu} for more details regarding comparison of the behaviour at nodes of quadratic differentials and $3/2$ differentials.
\end{remark}

A consequence is that the volume defined via a finite measure on $\modm_{g,\sigma}^{\text{spin}}$   can be calculated using intersection numbers over $\overline{\modm}_{g,\sigma}^{\text{spin}}$.   Define $L_{p_i}\to\overline{\modm}_{g,n}^{\text{spin}}$ to be the line bundle with fibre the cotangent space of a marked point on the twisted curve.  Due to the non-trivial isotropy of the marked points, its first Chern class $\tilde{\psi}_i:=c_1(L_{p_i})$ satisfies
\begin{equation}   \label{psispin}
\tilde{\psi}_{n+1}=\frac12p^*\psi_{n+1}.
\end{equation} 
Here $\psi_i:=c_1(\cl_{p_i})$ where $\cl_{p_i}\to\overline{\modm}_{g,n}$ is the line bundle with fibre the cotangent space of the marked point $p_i\in C$.  We also define $\kappa_1=\pi_*\psi_{n+1}^2\in H^2(\overline{\modm}_{g,n},\bq)$ and, using the same notation, $\kappa_1=p^*\kappa_1\in H^2(\overline{\modm}_{g,n}^{\text{spin}},\bq)$.

\begin{cor}  \label{volint}
\begin{equation}   \label{eq:volint}
\widehat{V}_{g,n}^{(m)}(L_1,...,L_n)=(-1)^{n-\frac12m}\epsilon_{g,\sigma}\int_{\overline{\modm}_{g,\sigma}^{\text{spin}}}c_{\text{top}}(\bar{E}_{g,\sigma})\exp\left\{2\pi^2\kappa_1+\sum_{i=1}^n L_i^2\tilde{\psi}_i\right\}
\end{equation}
where $\sigma=(1^n,0^m)$ and $c_{\text{top}}=c_{2g-2+n-\frac12m}$. 
\end{cor}
\begin{proof}
By \eqref{nsdef},
$\widehat{V}_{g,n}^{(m)}(L_1,...,L_n):=\epsilon_{g,\sigma}\int_{\modm_{g,\sigma}^{\rm spin}}e(E^\vee_{g,\sigma})\exp\omega(L_1,...,L_n,0^m)$ for $\sigma=(1^n,0^m)$.  
The  deformation $\omega(L_1,...,L_n,0^m)$  extends to a 2-form $\widehat{\omega}(L_1,...,L_n,0^m)$ on the moduli space of stable curves by an argument of Wolpert \cite{WolHom}.    
It is proven by Mirzakhani \cite{MirWei} via a general property of symplectic reductions, that %the cohomology class represented by its extension is
\[\big[\widehat{\omega}(L_1,...,L_n,0^m)\big]=2\pi^2\kappa_1+\frac12\sum_{i=1}^n L_i^2\psi_i\in H^2(\overline{\modm}_{g,n},\br).
\]  
This class pulls back under $p$ to $2\pi^2\kappa_1+\sum_{i=1}^n L_i^2\tilde{\psi}_i\in H^2(\overline{\modm}_{g,\sigma}^{\text{spin}},\br)$ since $\kappa_1$ is defined over $\overline{\modm}_{g,n}^{\text{spin}}$ via pullback.  The product of the extension of $e(E_{g,\sigma})$ to $\overline{\modm}_{g,\sigma}^{\text{spin}}$  with $\exp\{2\pi^2\kappa_1+\sum_{i=1}^n L_i^2\tilde{\psi}_i\}$ has the same integral as its restriction to $\modm_{g,\sigma}^{\text{spin}}$.  The factor $(-1)^{n-\frac12m}=(-1)^{\text{rank}(E_{g,\sigma})}$ is due to the presence of the dual bundle.
\end{proof}
From \eqref{volint} we see that the volume is a polynomial in $L_i$.  The degree of the polynomial is complementary to $c_{\text{top}}(\bar{E}_{g,\sigma})\in H^{4g-4+2n-m}$,
hence
\[\deg V_{g,n}^{(m)}(L_1,...,L_n)=2g-2+m.
\]

\begin{remark}
The symplectic form $\omega(L_1,...,L_n)$ can be naturally defined over the moduli space $\modm_{g,n}(L_1,...,L_n)\cong\modm_{g,n}$ 
of hyperbolic surfaces with $n$ geodesic boundary components of lengths $L_1,...,L_n$ so that $\omega(0,...,0)=\omega^{WP}$.  The diffeomorphism $\displaystyle f:\modm_{g,n}(0^{n})\overset{\cong}{\longrightarrow}
\modm_{g,\sigma}(L_1,...,L_n)$ allows one to work with deformations of a symplectic form over a single space, as we do here. We work only over $\modm_{g,\sigma}(0^n)$, on which the Euler form is defined, and deform the symplectic form.  %For $\sigma\in\{0,1\}^n$, the Euler form $e(E_{g,\sigma})$ is defined over $\modm_{g,\sigma}^{\text{spin}}(L_1,...,L_n)$ via the diffeomorphism $f$ lifted to the moduli space of spin curves.  
It would be desirable to have a definition of the Euler form directly over $\modm_{g,\sigma}^{\text{spin}}(L_1,...,L_n)$, as achieved by Wolpert \cite{WolHom} for the Weil-Petersson symplectic form.
\end{remark}

\subsubsection{Spin cohomological field theory}  \label{spincohft}
To calculate intersection numbers it is useful to use the structure of a cohomological field theory and more generally any collection of cohomology classes with well-behaved restriction properties to the strata:
\[
\phi_{\text{irr}}:\overline{\modm}_{g-1,n+2}\to\overline{\modm}_{g,n},\quad \phi_{h,I}:\overline{\modm}_{h,|I|+1}\times\overline{\modm}_{g-h,|I^c|+1}\to\overline{\modm}_{g,n},\  I\subset\{1,...,n\}.
\]
We also write
\[ D_I:=\overline{\modm}_{0,|I|+1}\times\overline{\modm}_{g,|I^c|+1}\to\overline{\modm}_{g,n}.
\]
A {\em cohomological field theory} (CohFT) is a pair $(H,\eta)$ which consists of a finite-dimension\-al complex vector space $H$ equipped with a non-degenerate symmetric bilinear form $\eta$ and a sequence of $S_n$-equivariant maps: 
\[ \Omega_{g,n}:H^{\otimes n}\to H^*(\overline{\modm}_{g,n},\bc)\]
that satisfy compatibility conditions from inclusion of strata:
%\[ \overline{\modm}_{g_1,n_1+1}\times\overline{\modm}_{g_2,n_2+1}\to\overline{\modm}_{g_1+g_2,n_1+n_2}\]
%$\displaystyle I_{g,n}(\alpha_1\otimes\cdots\otimes\alpha_n)=I_{g_1,n_1+1}\otimes I_{g_2,n_2+1}\big(\bigotimes_{j\in S_1}\alpha_j\otimes\Delta\otimes \bigotimes_{j\in S_2}\big)$ where $\Delta\in V\otimes V$ is dual to $\eta$.  
%then for $\displaystyle\Delta=\sum_{\alpha,\beta}\eta^{\alpha\beta}e_{\alpha}\otimes e_{\beta}$
\begin{equation}\label{glue}
\begin{split}
\phi_{\text{irr}}^*\Omega_{g,n}(v_1\otimes...\otimes v_n)&=\Omega_{g-1,n+2}(v_1\otimes...\otimes v_n\otimes\Delta) \\
\phi_{h,I}^*\Omega_{g,n}(v_1\otimes...\otimes v_n)&=\Omega_{h,|I|+1}\otimes \Omega_{g-h,|I^c|+1}\big(\bigotimes_{i\in I}v_i\otimes\Delta\otimes\bigotimes_{j\in I^c}v_j\big)
\end{split}
\end{equation}
where $\Delta\in H\otimes H$ is dual to $\eta\in H^*\otimes H^*$. There exists a vector $\un\in H$ satisfying 
\begin{equation} \label{unmet}
\Omega_{0,3}(v_1\otimes v_2\otimes \un)=\eta(v_1,v_2).
\end{equation}
%When $n=0$, $\Omega_g:=\Omega_{g,0}\in H^*(\overline{\modm}_{g},\bc)$.   A CohFT defines a product $\cdot$ on $V$ using the non-degeneracy of $\eta$ by \begin{equation}  \label{prod}  \eta(a\cdot b,c)=\Omega_{0,3}(a,b,c). \end{equation} and $\un$ is a unit for the product.

Using $\overline{\modm}_{g,n}^{\text{spin}}$, we define an example of a CohFT.
Identify the set $\{0,1\}$ with the set of distinct vectors $\{e_0,e_1\}\subset H\cong\bc^2$ which induces $\{0,1\}^n\to H^{\otimes n}$ where $\sigma\mapsto e_\sigma=e_{\sigma_1}\otimes...\otimes e_{\sigma_n}$ and define $\eta$ by $\eta(e_i,e_j)=\frac12\delta_{ij}$.  
Define
\[\Omega^{\text{spin}}_{g,n}(e_{\sigma_1}\otimes...\otimes e_{\sigma_n}):=p_*c(\bar{E}_{g,\sigma})\in H^{*}(\overline{\modm}_{g,n},\bq)\]
where recall that $p:\overline{\modm}_{g,n}^{\rm spin}\to\overline{\modm}_{g,n}$ forgets the spin and orbifold structures.
These classes, known as Chiodo classes \cite{ChiTow}, form a CohFT on $(V,\eta)$.  This is proven in \cite{LPSZChi} using an expression for $p_*c(\bar{E}_{g,\sigma}^\vee)$ in terms of sums over stable graphs proven in \cite{JPPZDou}.  It is proven directly in \cite{NorNew} via pullback properties of the bundles $\bar{E}_{g,\sigma}$.

For $e_{\sigma}=e_{\sigma_1}\otimes ...\otimes e_{\sigma_n}$, define
\begin{equation}   \label{Omegadef}
\Omega_{g,n}(e_{\sigma}):=\epsilon_{g,\sigma}p_*c_{\text{top}}(\bar{E}_{g,\sigma}^\vee)%\Omega^{\text{spin}}_{g,n}(e_{\sigma})^{\text{[top]}}
\end{equation}
where $c_{\text{top}}=c_{2g-2+\frac12(n+|\sigma|)}$ and $\epsilon_{g,\sigma}=2^{g-1+\frac12(n+|\sigma|)}$.  These classes satisfy \eqref{glue} for $\Delta=e_1\otimes e_1$.  They produce a degenerate bilinear form, so they do not define precisely a CohFT.
The pushforward of the classes $\Omega_{g,n}(e_{\sigma})\in H^{*}(\overline{\modm}_{g,n},\bq)$ under the map that forgets the Ramond points is used in \cite{CGGRel} to produce relations among polynomials in the classes $\kappa_m\in H^{*}(\overline{\modm}_{g,n},\bq)$ conjectured in \cite{KNoPol}.
    
The forgetful map defined in \eqref{forgetp}  extends to $\pi:\overline{\modm}_{g,\sigma'}^{\text{spin}}\to\overline{\modm}_{g,\sigma}^{\text{spin}}$ for $\sigma'=(\sigma,1)\in\{0,1\}^{n+1}$.
The following exact sequence is proven in \cite{NorNew}:
 \[0\to\xi\to\bar{E}_{g,\sigma'}\to\pi^*\bar{E}_{g,\sigma}\to 0
\]
where
\[c(\xi)=\frac{1}{c(L_{p_{n+1}})}=1-\tilde{\psi}_{n+1}.
\]
Push forward by $p$ the relation $c(\bar{E}_{g,\sigma'})=c(\xi)\pi^*c(\bar{E}_{g,\sigma})$ to get
\begin{equation} \label{forgbund}
\Omega^{\text{spin}}_{g,n+1}(e_1,v_1,...,v_n)=(1-\tfrac12\psi_{n+1})\pi^*\Omega^{\text{spin}}_{g,n}(v_1,...,v_n)
\end{equation}
which uses $\tilde{\psi}_{n+1}=\frac12p^*\psi_{n+1}$.
The factor $\epsilon_{g,\sigma}$ in \eqref{Omegadef} and $-1$ for the dual bundle together with \eqref{forgbund}  yields:
\begin{equation} \label{forgbund1}
\Omega_{g,n+1}(e_1,v_1,...,v_n)=\psi_{n+1}\pi^*\Omega_{g,n}(v_1,...,v_n).
\end{equation}

The expression for $\widehat{V}_{g,n}^{(m)}(L_1,...,L_n)$ as an intersection number on $\overline{\modm}_{g,\sigma}^{\text{spin}}$ for $\sigma=(1^n,0^m)$ given in \eqref{eq:volint}, now becomes an intersection number over $\overline{\modm}_{g,n+m}$.
\[
\widehat{V}_{g,n}^{(m)}(L_1,...,L_n)=\int_{\overline{\modm}_{g,n+m}}\hspace{-5mm}\Omega_{g,n+m}(e_1^{\otimes n}\otimes e_0^{\otimes m})\exp\left\{2\pi^2\kappa_1+\frac12\sum_{i=1}^n L_i^2\psi_i\right\}
\]
and
$$\widehat{V}_{g,n}\WP(s,L_1,...,L_n)=\sum_{m=0}^\infty\frac{s^m}{m!}\widehat{V}^{(m)}_{g,n}(L_1,...,L_n)
$$
is a generating function for intersection numbers.
The pullback relation \eqref{forgbund1} can be encoded in $\widehat{V}_{g,n}\WP$ by:
\begin{equation}  \label{dilaton}
 \widehat{V}_{g,n+1}\WP(s,L_1,...,L_n,2\pi i)=\left(2g-2+n+s\frac{\partial}{\partial s}\right)\widehat{V}_{g,n}\WP(s,L_1,...,L_n)
\end{equation}
by generalising an argument in \cite{DNoWei}.

The symmetry of the intersection numbers involving $\Omega_{g,n+m}(e_1^{\otimes n}\otimes e_0^{\otimes m})$ allows them to be naturally expressed using the following tau notation:
\[\langle\tau_{k_1}...\tau_{k_n}\nu^m\rangle_g:=\int_{\overline{\modm}_{g,n+m}}\hspace{-5mm}\Omega_{g,n+m}(e_1^{\otimes n}\otimes e_0^{\otimes m})\exp\left\{2\pi^2\kappa_1\right\}\prod_{i=1}^n\psi_i^{k_i}.
\]
They can be stored in a partition $\cz(\hbar,s,\vec{t})$ defined via:
\begin{equation}  \label{Zpart}
\cz(\hbar,s,\vec{t})=\exp\sum_{g,n,m}\frac{\hbar^{g-1}s^m}{n!\ m!}\sum_{\vec{k}\in\bn^n}\langle\tau_{k_1}...\tau_{k_n}\nu^m\rangle_g\prod_{i=1}^nt_{k_i}.
\end{equation}
Equivalently
\[ \left.\frac{\partial^n}{\partial t_{k_1}...\partial t_{k_n}}\log \cz(\hbar,s,\vec{t})\right|_{\vec{t}=0}=\sum_{g,n,m}\hbar^{g-1}\frac{s^m}{m!}\sum_{\vec{k}\in\bn^n}\langle\tau_{k_1}...\tau_{k_n}\nu^m\rangle_g.\]
It is related to the volumes via
\begin{equation}  \label{Zvol}
\cz(\hbar,s,\vec{t})=\exp\sum_{g,n}\frac{\hbar^{g-1}}{n!}\widehat{V}_{g,n}(s,L_1,...,L_n)|_{\{L_i^{2k}=2^kk!t_k\}}.
\end{equation}

Construct a partition function involving simpler intersection numbers as follows.
Remove the $\exp\left\{2\pi^2\kappa_1\right\}$ term in the definition of $\langle\tau_{k_1}...\tau_{k_n}\nu^m\rangle_g$ above to define 
\[\langle\overline{\tau}_{k_1}...\overline{\tau}_{k_n}\nu^m\rangle_g:=\int_{\overline{\modm}_{g,n+m}}\hspace{-5mm}\Omega_{g,n+m}(e_1^{\otimes n}\otimes e_0^{\otimes m})\prod_{i=1}^n\psi_i^{k_i}
\]
and its associated partition function
\[\bar{\cz}(\hbar,s,\vec{t})=\exp\sum_{g,n,m}\frac{\hbar^{g-1}s^m}{n!\ m!}\sum_{\vec{k}\in\bn^n}\langle\overline{\tau}_{k_1}...\overline{\tau}_{k_n}\nu^m\rangle_g\prod_{i=1}^nt_{k_i}.
\]

The following theorem relates the partition functions $\cz(\hbar,s,\vec{t})$ and $\bar{\cz}(\hbar,s,\vec{t})$ by translation of variables.  This is analogous to a result of Manin and Zograf \cite{MZoInv} for intersection numbers of $\psi$ classes and $\kappa$ classes over $\overline{\modm}_{g,n}$.  The $s=0$ case was proven in \cite{NorEnu}.
\begin{thm} \label{th:transl}
\begin{equation}  \label{transl}
\cz(\hbar,s,\vec{t})=\bar{\cz}(\hbar,s,t_0,t_1+2\pi^2,t_2-2\pi^4,...,t_k+\frac{(-1)^{k+1}}{k!}(2\pi^2)^k,...)
\end{equation} 
%\[Z(\hbar,s,u,\vec{t})=Z(\hbar,s,0,t_0,t_1+u,t_2-\frac{1}{2!}u^2,...,t_k-\frac{(-1)^k}{k!}u^k,...)\]
\end{thm}
\begin{proof}
The proof of the translation formula \cite{MZoInv} by Manin and Zograf uses the pushforward relation $\pi_*\psi_{n+1}^m=\kappa_{m-1}$.  The following key pushforward relation follows from \eqref{forgbund1}:
$$\pi_*(\psi_{n+1}^m\Omega_{g,n+1}(\sigma,1))=\pi_*(\psi_{n+1}^{m+1}\cdot\pi^*\Omega_{g,n}(\sigma))=\kappa_m\Omega_{g,n}(\sigma).
$$
It allows us to adapt much of the proof of Manin and Zograf.  The main tool used by Manin and Zograf is the following pushforward relation for $\kappa_1^N$ from \cite{KMZHig}  involving a sum over ordered partitions of $N$:
\begin{equation}  \label{KMZ}
\frac{\kappa_1^N}{N!}=\pi_*\left(\sum_{\mu\vdash N}\frac{(-1)^{N+\ell(\mu)}}{\ell(\mu)!}\prod_{j=n+1}^{n+\ell(\mu)}\frac{\psi_j^{\mu_j+1}}{\mu_j!}\right)
\end{equation} 
where $\mu\vdash N$ is an ordered partition of $N$ of length $\ell(\mu)$ and $\pi_*:\overline{\modm}_{g,\ell(\mu)}\to\overline{\modm}_{g}$ is the Gysin homomorphism induced by the map that forgets the marked  points.  From \eqref{KMZ}  we also have
\[
\frac{\kappa_1^N}{N!}\prod_{j=1}^n\psi_j^{k_j}=\pi_*\left(\sum_{\mu\vdash N}\frac{(-1)^{N+\ell(\mu)}}{\ell(\mu)!}\prod_{j=n+1}^{n+\ell(\mu)}\frac{\psi_j^{\mu_j+1}}{\mu_j!}\prod_{j=1}^n\psi_j^{k_j}\right)
\]
where $\pi_*:\overline{\modm}_{g,n+\ell(\mu)}\to\overline{\modm}_{g,n}$ forgets the last $\ell(\mu)$ points, since the factor $\displaystyle\prod_{j=1}^n\psi_j^{k_j}$ can be replaced by its pullback  via 
\[\prod_{j={n+1}}^{n+\ell(\mu)}\psi_j\cdot\prod_{j=1}^n\psi_j^{k_j}=\prod_{j={n+1}}^{n+\ell(\mu)}\psi_j\cdot\pi^*\prod_{j=1}^n\psi_j^{k_j}\]
and then brought outside of the pushforward.
When $\Omega_{g,n}(\sigma)$ is present, we have
\[
\Omega_{g,n}(\sigma)\frac{\kappa_1^{N}}{N!}=\pi_*\left(\Omega_{g,n+\ell(\mu)}(\sigma,1^{\ell(\mu)})\sum_{\mu\vdash N}\frac{(-1)^{N+\ell(\mu)}}{\ell(\mu)!}\prod_{j=n+1}^{n+\ell(\mu)}\frac{\psi_j^{\mu_j}}{\mu_j!}\right)
\]
since $\Omega_{g,n+\ell}(\sigma,1^\ell)=\psi_{n+1}...\psi_{n+\ell}\pi^*\Omega_{g,n}(\sigma)$ so that $\psi_j^{\mu_j+1}$ in the right hand side of \eqref{KMZ} is replaced by $\psi_j^{\mu_j}$.  The product of $\psi$ classes inside $\Omega_{g,n+\ell(\mu)}(\sigma,1^{\ell(\mu)})$ again allows us to replace $\prod_{j=1}^n\psi_j^{k_j}$ by $\pi^*\prod_{j=1}^n\psi_j^{k_j}$  to deduce:
\[
\Omega_{g,n}(\sigma)\frac{\kappa_1^{N}}{N!}\prod_{j=1}^n\psi_j^{k_j}=\pi_*\left(\Omega_{g,n+\ell(\mu)}(\sigma,1^{\ell(\mu)})\sum_{\mu\vdash N}\frac{(-1)^{N+\ell(\mu)}}{\ell(\mu)!}\prod_{j=n+1}^{n+\ell(\mu)}\frac{\psi_j^{\mu_j}}{\mu_j!}\prod_{j=1}^n\psi_j^{k_j}\right).
\]
Hence
\[
\int_{\overline{\modm}_{g,n}}\hspace{-3mm}\Omega_{g,n}(\sigma)\frac{\kappa_1^{N}}{N!}\prod_{j=1}^n\psi_j^{k_j}\hspace{-1mm}=\hspace{-1mm}\sum_{\mu\vdash N}\frac{(-1)^{N+\ell(\mu)}}{\ell(\mu)!}\int_{\overline{\modm}_{g,n+\ell(\mu)}}\hspace{-9mm}\Omega_{g,n+\ell(\mu)}(\sigma,1^{\ell(\mu)})\prod_{j=n+1}^{n+\ell(\mu)}\frac{\psi_j^{\mu_j}}{\mu_j!}\prod_{j=1}^n\psi_j^{k_j}\hspace{-1mm}.
\]
Since the right hand side can be obtained from coefficients of $\bar{\cz}(\hbar,s,\vec{t})$, this shows that $\cz(\hbar,s,\vec{t})$ is uniquely determined from $\bar{\cz}(\hbar,s,\vec{t})$ and it remains to prove that it is obtained via translation.

Let $\bar{\cf}(\hbar,s,\vec{t})=\log\bar{\cz}(\hbar,s,\vec{t})=\sum\hbar^{g-1}\cf_g(s,\vec{t})$ and consider the translation of variables $t_k\mapsto t_k+\frac{(-1)^{k+1}}{k!}u^k$ where we later substitute $u=2\pi^2$ to produce the right hand side of \eqref{transl}.  
\begin{align*}
\bar{\cf}_g(s,\hbar,t_0&,t_1+u,t_2-\frac{1}{2!}u^2,...)\\
&=\sum_{n,m,\vec{k}}\frac{s^m}{m!}\int_{\overline{\modm}_{g,n+m}}\hspace{-5mm}\Omega_{g,n+m}(e_1^{\otimes n}\otimes e_0^{\otimes m})\prod_{j=1}^n\psi_j^{k_j}\left(t_{k_j}+\frac{(-1)^{k_j+1}}{k_j!}u^{k_j}\right)
\end{align*}
The translation of $\bar{\cf}(\hbar,s,\vec{t})$ does not mix different powers of $s$ so each coefficient of $s^m$ can be treated separately.  The coefficient of %$\displaystyle\prod_{j=1}^nt_{k_j}$ 
%Define \[ F_g(u,\hbar,\vec{t})=\sum_n\frac{1}{n!}\sum_{\vec{k}\in\bn^n}\int_{\overline{\modm}_{g,n}}\hspace{-3mm}\Omega_{g,n}(1^n)\exp(u\kappa_1)\prod_{j=1}^n\psi_j^{k_j}t_{k_j}.\]
$\frac{s^m}{m!}\prod_{j=1}^nt_{k_j}$ in $\bar{\cf}_g(s,\hbar,t_0,t_1+u,t_2-\frac{1}{2!}u^2,...)$ is
\[\sum_{\mu}\binom{n+\ell(\mu)}{n}\prod_{j=1}^{\ell(\mu)}\frac{(-1)^{\mu_j+1}u^{\mu_j}}{\mu_j!(n+\ell(\mu))!}\int_{\overline{\modm}_{g,n+m+\ell(\mu)}}\hspace{-12mm}\Omega_{g,n+m+\ell(\mu)}(1^{n+\ell(\mu)},0^m)\prod_{j=1}^n\psi_j^{k_j}\hspace{-2mm}\prod_{j=n+1}^{n+\ell(\mu)}\psi_j^{\mu_j}
\]
where the sum is over all $\mu=(\mu_1,....,\mu_k)$ for all $k$.
The factor $\binom{n+\ell(\mu)}{n}$ is needed because $\prod_{j=1}^n\psi_j^{k_j}$ has been written in the first $n$ factors, but arises in any $n$ factors taken from $n+\ell(\mu)$ factors.

Collect those $\mu$ that are ordered partitions of $N$,  to obtain the $u^N$ terms of $\bar{\cf}_g(s,\hbar,t_0,t_1+u,t_2-\frac{1}{2!}u^2,...)$:
\begin{align*}
\sum_{\mu}&\binom{n+\ell(\mu)}{n}\prod_{j=1}^{\ell(\mu)}\frac{(-1)^{\mu_j+1}u^{\mu_j}}{\mu_j!(n+\ell(\mu))!}\int_{\overline{\modm}_{g,n+m+\ell(\mu)}}\hspace{-12mm}\Omega_{g,n+m+\ell(\mu)}(1^{n+\ell(\mu)},0^m)\prod_{j=1}^n\psi_j^{k_j}\hspace{-2mm}\prod_{j=n+1}^{n+\ell(\mu)}\psi_j^{\mu_j}\\
&=\sum_Nu^N\sum_{\mu\vdash N}\frac{(-1)^{N+\ell(\mu)}}{\ell(\mu)!}\frac{1}{n!}\int_{\overline{\modm}_{g,n+m+\ell(\mu)}}\hspace{-10mm}\Omega_{g,n+m+\ell(\mu)}(1^{n+\ell(\mu)},0^m)\prod_{j=1}^n\psi_j^{k_j}\hspace{-2mm}\prod_{j=n+1}^{n+\ell(\mu)}\frac{\psi_j^{\mu_j}}{\mu_j!}\\
&=\sum_N\frac{1}{n!}\int_{\overline{\modm}_{g,n+m}}\hspace{-1mm}\Omega_{g,n+m}(\sigma)\frac{(u\kappa_1)^{N}}{N!}\prod_{j=1}^n\psi_j^{k_j},\qquad\sigma=(1^n,0^m).
\end{align*}
Set $u=2\pi^2$, then we have proven that the coefficient of $(2\pi^2)^N\frac{s^m}{m!}\prod_{j=1}^nt_{k_j}$ in $\bar{\cf}_g(s,\hbar,t_0,t_1+2\pi^2,t_2-\frac{1}{2!}4\pi^4,...)$ equals the coefficient of $(2\pi^2)^N\frac{s^m}{m!}\prod_{j=1}^nt_{k_j}$ in $\cf_g(s,\hbar,t_0,t_1,...)$ proving their equality.
\end{proof}

\subsection{The disk function}
We now give a proof of Theorem~\ref{main}.  Let $\ct(f(L)):=\int_0^\infty e^{-zL}f(L)dL$ denote the Laplace transform and define:
\[ F(z):=\ct(L\widehat{V}_{0,1}(s,L)).\]
For $P(x,y)$ an odd polynomial in $x$ and $y$, it is proven in \cite{NorEnu} that:
\[\int_0^\infty\hspace{-2mm}\int_0^\infty\hspace{-2mm} \frac{x^{2i+1}}{(2i+1)!!} \frac{y^{2j+1}}{(2j+1)!!}D(L,x,y)dxdy=\hspace{-4mm}\sum_{m+n=i+j+1}\hspace{-1mm}\frac{L^{2m +1}}{(2m +1)!}a_n\]
for $a_n$ defined by $\frac{1}{\cos{2\pi x}}=\sum_na_nx^{2n}.$
This leads to
\[
\ct\left\{\int_0^\infty\hspace{-2mm}\int_0^\infty dxdy D(L,x,y)P(x,y)\right\}=\left[\frac{1}{\cos(2\pi z)}\ct\{P\}(z,z)\right]_{z=0}.
\]
where $[\cdot]_{z=0}$ denotes the principal part of a function at $z=0$.    Hence
the recursion relation in Theorem~\ref{main} is equivalent to:
\begin{equation}  \label{Frec} F(z)=\frac{s^2}{2z^2}+\left[\frac{F(z)^2}{2\cos(2\pi z)}\right]_{z=0}
\end{equation}
which we prove below.
This allows for an easy expansion of $F(z)$:
\[ F(z)=\frac{s^2}{2z^2}+\frac{s^4}{8}\left(\frac{1}{z^4}+\frac{2\pi^2}{z^2}\right)+\frac{s^6}{48}\left(\frac{3}{z^6}+\frac{12\pi^2}{z^4}+\frac{22\pi^4}{z^2}\right)+O(s^8)
\]
which produces the series for $ \widehat{V}_{0,1}\WP(s,L)$ in the introduction.

To prove \eqref{Frec}, with the help of Theorem~\ref{th:transl}, define instead the volume associated to the partition function $\bar{\cz}(\hbar,s,\vec{t})$:
\begin{equation}  
\widehat{V}_{g,n}(s,L_1,...,L_n):=\sum_{m=0}^\infty\frac{s^m}{m!}\int_{\overline{\modm}_{g,n+m}}\hspace{-5mm}\Omega_{g,n+m}(e_1^{\otimes n}\otimes e_0^{\otimes m})\exp\left\{\frac12\sum_{i=1}^n L_i^2\psi_i\right\}.
\end{equation}
which gives the top degree terms of $\widehat{V}_{g,n}\WP(L_1,...,L_n)$ for each coefficient of $s^{2m}$.  Furthermore, the recursion \eqref{reconj} restricts to these top degree terms.  Define
\[ G(z):=\ct(L\widehat{V}_{0,1}(s,L)).\]
The top degree part of the disk recursion in Theorem~\ref{main} is equivalent to the recursion:
\[ G(z)=\frac{s^2}{2z^2}+\left[\frac12G(z)^2\right]_{z=0}
\]
proven by replacing the $1/\cos(2\pi z)$ term above by $1$.  It leads to the expansion:
\[ G(z)=\frac{s^2}{2z^2}+\frac{s^4}{8z^4}+\frac{s^6}{16z^6}+O(s^8)
\]
\begin{thm}  \label{th:Grec}
The disk function $G(z):=\cl(L\widehat{V}_{0,1}(s,L))$ is uniquely determined by the recursion
\begin{equation}  \label{Grec} 
G(z)=\frac{s^2}{2z^2}+\left[\frac12G(z)^2\right]_{z=0}.
\end{equation}
\end{thm}
\begin{proof}
Consider the map $\pi^{(2m)}:\overline{\modm}_{0,2m+3}\to\overline{\modm}_{0,3}$ that forgets all but three marked points.  Then
%\[\int_{\overline{\modm}_{0,2m+3}}\Omega_{0,2m+3}(e_1\otimes e_0^{\otimes 2m+2})\]
\[\psi_1=(\pi^{(2m)})^*\psi_1+\sum_{\underset{2,3\in I^c}{1\in I}}D_I=\sum_{\underset{2,3\in I^c}{1\in I}}D_I\]
since $(\pi^{(2m)})^*\psi_1=0$.  Using $\displaystyle\psi_1^m=\psi_1^{m-1}\sum_{\underset{2,3\in I^c}{1\in I}}D_I$ for $m>0$, we have
\begin{align*}
\int_{\overline{\modm}_{0,2m+3}}\hspace{-2mm}\psi_1^m\Omega_{0,2m+3}(e_1\otimes e_0^{\otimes 2m+2})&=\sum_{\underset{2,3\in I^c}{1\in I}}\int_{\overline{\modm}_{0,2m+3}}\hspace{-5mm}\psi_1^{m-1}D_I\Omega_{0,2m+3}(e_1\otimes e_0^{\otimes 2m+2}).
\end{align*}
The summand is given by
\begin{align*}
\int_{\overline{\modm}_{0,2m+1}}\hspace{-5mm}&\psi_1^{m-1}D_I\Omega_{0,2m+3}(e_1\otimes e_0^{\otimes 2m+2})\\
&=\int_{\overline{\modm}_{0,|I|+1}}\hspace{-5mm}\psi_1^{m-1}\Omega_{0,|I|+1}(e_1\otimes e_0^{\otimes (|I|-1)}\otimes e_1)\int_{\overline{\modm}_{0,|I^c|+1}}\hspace{-5mm}\Omega_{0,|I^c|+1}(e_1\otimes e_0^{\otimes |I^c|}).
\end{align*}
For $\sigma\in\{0,1\}^n$, 
\[\deg\Omega_{0,n}(e_\sigma)=\frac12(n+|\sigma|)-2= n-3\qquad\Leftrightarrow \qquad|\sigma|=n-2\] 
hence the second factor in the summand vanishes unless $|I^c|=2$.  So $I^c=\{2,3\}$ and only one summand is non-zero, leaving
\begin{align*}
\int_{\overline{\modm}_{0,2m+3}}\hspace{-2mm}\psi_1^m\Omega_{0,2m+3}&(e_1\otimes e_0^{\otimes 2m+2})\\
&=\int_{\overline{\modm}_{0,2m+2}}\hspace{-5mm}\psi_1^{m-1}\Omega_{0,2m+2}(e_1^{\otimes 2}\otimes e_0^{\otimes 2m})\int_{\overline{\modm}_{0,3}}\hspace{-5mm}\Omega_{0,|I^c|+1}(e_1\otimes e_0^{\otimes 2})\\
&=\int_{\overline{\modm}_{0,2m+2}}\hspace{-5mm}\psi_1^{m-1}\Omega_{0,2m+2}(e_1^{\otimes 2}\otimes e_0^{\otimes 2m})\\
&=\int_{\overline{\modm}_{0,2m+2}}\hspace{-5mm}\psi_1^{m-1}\psi_2\pi^*\Omega_{0,2m+1}(e_1\otimes e_0^{\otimes 2m})\\
&=\int_{\overline{\modm}_{0,2m+2}}\hspace{-5mm}\psi_2\pi^*\big(\psi_1^{m-1}\Omega_{0,2m+1}(e_1\otimes e_0^{\otimes 2m})\big)\\
&=(2m-1)\int_{\overline{\modm}_{0,2m+1}}\hspace{-5mm}\psi_1^{m-1}\Omega_{0,2m+1}(e_1\otimes e_0^{\otimes 2m})\\
&=(2m-1)!!
\end{align*}
where the first equality uses the restriction of $\Omega_{0,2m+3}$ to a boundary divisor, the second equality uses $\Omega_{0,|I^c|+1}(e_1\otimes e_0^{\otimes 2})=1$ from Lemma~\ref{initial}, the third equality uses \eqref{forgbund1}, the fourth equality uses $\psi_2\psi_m=\psi_2\pi^*\psi_m$ where $\pi$ is the forgetful map that forgets the 2nd point, the fifth equality uses pushforward and the sixth equality uses induction.

Hence
\[ \widehat{V}_{0,1}(s,L)=\sum_{m=0}^\infty\frac{(2m-1)!!}{m!2^m}L^{2m}\frac{s^{2m+2}}{(2m+2)!}
\]
thus
\[ G(z)=\cl(L\widehat{V}_{0,1}(s,L))=\sum_{m=0}^\infty\frac{(2m-1)!!}{m!2^m(2m+2)}\frac{s^{2m+2}}{z^{2m+2}}=\sum_{m=0}^\infty\frac{C_m}{2^{2m+1}}\frac{s^{2m+2}}{z^{2m+2}}
\]
where $\displaystyle C_m=\frac{1}{m+1}\binom{2m}{m}$ is the $m$th Catalan number.

The recursion \eqref{Grec} uniquely determines a series 
\[g(z)=\sum g_m\frac{s^{2m}}{z^{2m}}=\frac{s^2}{2z^2}+\frac{s^4}{8z^4}+\frac{s^6}{16z^6}+O(s^8).\]  
We prove that $g(z)$ coincides with $G(z)$ by induction.  The initial coefficient is $G(z)=\frac{C_0}{2}\frac{s^2}{z^2}+O(s^4)=\frac{s^{2}}{2z^{2}}+O(s^4)$ which agrees with the initial coefficient of $g(z)$.  Assume that the coefficients of $g(z)$ agree with the coefficients of $G(z)$ up to the terms $\left(\frac{s}{z}\right)^{2m}$, i.e. 
\[g_k= \frac{C_{k-1}}{2^{2k-1}},\quad k\leq m\]
Then the coefficient of $\left(\frac{s}{z}\right)^{2m+2}$ in $g(z)$ equals the coefficient of $\left(\frac{s}{z}\right)^{2m+2}$ in $\frac12g(z)^2$ which is given by:
\[ g_{m+1}=\frac12\sum_{k=1}^{m}g_kg_{m+1-k}=\frac{1}{2^{2m+1}}\sum_{k=1}^{m}C_{k-1}C_{m-k}= \frac{C_m}{2^{2m+1}}
\]
hence by induction $g_k= \frac{C_{k-1}}{2^{2k-1}}$ for all $k$ and we conclude that $G(z)$ satisfies the recursion \eqref{Grec}.
\end{proof}

\begin{cor}
The disk function $F(z)$ satisfies $\eqref{Frec}$, hence Theorem~\ref{main} holds.
\end{cor}
\begin{proof}
Theorem~\ref{th:Grec} shows that the top degree part of the recursion for $\widehat{V}_{0,1}\WP(s,L)$ in Theorem~\ref{main} holds.  Theorem~\ref{th:transl} shows that the lower degree terms of $\widehat{V}_{0,1}\WP(s,L)$ are determined by the top degree terms via translation of the associated partition functions.  It is proven in \cite{NorEnu} that the translation also sends the recursion for the top degree terms, equivalently \eqref{Grec}, to the recursion for the entire polynomial, equivalently \eqref{Frec}.  The same proof applies here, because it involves only the kernels $D(x,y,z)$ and $R(x,y,z)$ and their associated linear transformations on polynomials.
\end{proof}
%$\langle\tau_m\tau_{\bar{0}}^{2m+2}\rangle&=(2m-1)!!2^{-m}$

\section{Virasoro structure}   \label{virastr}

The underlying algebraic, or integrable, structures of partition functions storing intersection numbers over the moduli spaces of marked, stable curves began with the conjecture of Witten \cite{WitTwo}.  In this section we write the exact recursion \eqref{reconj} in terms of Virasoro operators acting on the partition function $\bar{\cz}(\hbar,s,\vec{t})$.  We relate this algebraic structure directly to the geometric structure of the intersection numbers, and in particular use it to prove Theorem~\ref{conj}.  An important intermediary construction is a partition function built from polynomials  in kappa classes which we relate to $\bar{\cz}(\hbar,s,\vec{t})$.

The following Virasoro operators were defined in \cite{IZuCom}, for $m\geq-1$:
\begin{align} \label{virop}
\cl_m:=\frac12\hbar\hspace{-2mm}\sum_{i+j=m-1}\hspace{-2mm}(2i+1)!!(2j+1)!!
\frac{\partial^2}{\partial t_i\partial t_j}+
\sum_{k-j=m}&\frac{(2k+1)!!}{(2j-1)!!}t_j\frac{\partial}{\partial t_k}\\
&+
\hbar^{-1}\frac{\delta_{m,-1}t_0^2}{2}+\frac{\delta_{m,0}}{8}.\nonumber
\end{align} 
They form half of the Virasoro algebra and satisfy:
\[ [\cl_m,\cl_n]=2(m-n)\cl_{m+n},\quad m,n\geq-1.
\]

\begin{proposition}{\cite{NorEnu}}  \label{vireq}
The recursion \eqref{reconj} is equivalent to the Virasoro constraints
\begin{equation}  \label{virconj}
\bigl((2m+1)!!\frac{\partial}{\partial t_m}-\cl_m-\tfrac{s^2}{2\hbar}\delta_{m,0}\bigr)\bar{\cz}(\hbar,s,\vec{t})=0.
\end{equation}
\end{proposition}
\begin{proof}
The statement holds up to all orders in $s$, not just $O(s^4)$.
The proof in \cite{NorEnu} treats the $s=0$ case but it immediately applies to the more general case.  It shows that the kernels $D(x,y,z)$ and $R(x,y,z)$ in \eqref{reconj} produce linear transformations on monomials $\prod_{i=1}^nL_i^{2m_i}$ which, after the substitution $L_i^{2m}=2^mm!t_m$, convert to the partial differential operators $\cl_n$ in $t_k$.  This identifies \eqref{virconj} with the top degree part of \eqref{reconj}.  The full recursion \eqref{reconj} is proven in \cite{NorEnu} to be a consequence of its top degree part, essentially via conjugating Virasoro operators by the translation operator, although conveniently expressed via topological recursion, which we do not repeat here.  There is one minor adjustment to the argument that must take into account the term $\frac{1}{2\hbar}t_0s^2$ in $\log\bar{\cz}(\hbar,s,\vec{t})$.  This affects only the operator $\cl_0$ and requires the extra term 
\[\frac{s^2}{2\hbar}=\frac{\partial}{\partial t_0}\frac{t_0s^2}{2\hbar}.\]
\end{proof}
\begin{remark}
The Virasoro constraint \eqref{virconj} for $m=0$ can be proven as a consequence of the pullback relation \eqref{forgbund1} as follows.   From \eqref{forgbund1} we get an expression for removal of a $\overline{\tau}_0$ term:
\[ \langle\overline{\tau}_0\prod_{j=1}^n\overline{\tau}_{k_j}\nu^m\rangle_g=(2g-2+n+m)\langle\prod_{j=1}^n\overline{\tau}_{k_j}\nu^m\rangle_g.
\]
The genus $g$ is uniquely determined by $\sum_1^n k_j=g-1+\frac12m$.  The differential operator $\sum_k(2k+1)t_k\frac{\partial}{\partial t_k}$, which forms part of $\cl_0$, multiplies the monomial $\prod_{j=1}^nt_{k_j}$ by $\sum_1^n (2k_j+1)=2g-2+m+n$ hence $\frac{\partial}{\partial t_0}-\cl_0$ achieves removal of a $\overline{\tau}_0$ term.  The extra terms $\frac18$ and $\frac{s^2}{2}$ in $\cl_0-\frac12s^2$ take into account the terms $\frac18t_0$ and $\frac{s^2t_0}{2}$.
\end{remark}

\subsection{Kappa class tau function}
Consider the following sequence of polynomials in kappa classes introduced in \cite{KNoPol}.   Define $\{s_i\in\bq\mid i>0\}$ by
$$\exp\left(-\sum_{i>0} s_it^i\right)=\sum_{k=0}^\infty(-1)^k(2k+1)!!t^k.
$$
Define
$$\bk=K_0+K_1+K_2+...=\exp(\sum s_i\kappa_i)\in H^*(\overline{\modm}_{g,n},\bq)$$
where $K_m\in H^{2m}(\overline{\modm}_{g,n},\bq)$.
This defines a sequence of homogeneous polynomials in $\kappa_m$.  The first few polynomials are given by:
$$K_1=3\kappa_1,\quad K_2=\frac32(3\kappa_1^2-7\kappa_2),\quad K_3=\frac32(3\kappa_1^3-21\kappa_1\kappa_2+46\kappa_3)$$
$$K_4=\frac98(3\kappa_1^4-42\kappa_1^2\kappa_2+49\kappa_2^2+184\kappa_1\kappa_3-562\kappa_4)
$$

\begin{thm} \cite{CGGRel}  \label{vanK}
The following polynomial relations among the $\kappa_j$ hold:
\[
K_m(\kappa_1,...,\kappa_m)=0\quad\text{for}\quad m>2g-2+n\quad\text{except}\quad (m,n)=(3g-3,0). 
\]
\end{thm}
Theorem~\ref{vanK} implies that the following functions are regular in $s$.  Define
\[\bk_{g,n}(s):=\sum_m s^{2(2g-2+n-m)}\bk_m\in H^*(\overline{\modm}_{g,n},\bq)\]
and a generating function
\[\cf^\bk(\hbar,s,t_0,t_1,t_2,\dots)=\sum_{g,n,\vec{k}}
\frac{\hbar^{g-1}}{n!}\sum_{n=0}^\infty\int_{\overline{\modm}_{g,n}}\bk_{g,n}(s) \prod_{i=1}^n\psi_i^{k_i}t_{k_i}.
\]

It is proven in \cite{KNoPol} that the partition function $\cz^\bk(\hbar,s,\vec{t})=\exp\cf^\bk(\hbar,s,\vec{t})$ satisfies the following Virasoro constraints
\begin{equation}\label{virK}
\bigl\{(2m+1)!!\frac{\partial}{\partial t_m}-\cl_m-s^2\cl_{m-1}\bigr\}\cz^\bk(\hbar,s,\vec{t})=0,\quad m=0,1,2,\dots
\end{equation}
where the Virasoro operators $\cl_m$ are defined in \eqref{virop}.
These equations determine $\cz^\bk(\hbar,s,\vec{t})$ uniquely and give another proof that it is regular in $s$, and also shows that $\cz^\bk(\hbar,s=0,\vec{t})$ is the BGW tau function define in  \cite{BGrExt,GWiPos}.

The relation of the polynomials $K_m(\kappa_1,...,\kappa_m)$ with the cohomology classes in this paper, and hence also the volumes, is given in the following theorem.
\begin{thm} \cite{CGGRel}  \label{pushK}
The following pushforward relations hold:
\[
\frac{1}{m!}\pi^{(m)}_*\Omega_{g,n+m}(e_1^{\otimes n}\otimes e_0^{\otimes m})=K_{2g-2+n-\frac12m}(\kappa_1,\kappa_2,...).
\]
\end{thm}

The partition function $\cz(\hbar,s,\vec{t})$, related to the volumes via \eqref{Zvol}, is obtained from the partition function $\bar{\cz}(\hbar,s,\vec{t})$ via the translation \eqref{transl}.  A further relation of $\bar{\cz}(\hbar,s,\vec{t})$ to the partition function $\cz^\bk(\hbar,s,\vec{t})$ is used to prove \eqref{virconj} up to $O(s^4)$, hence also Theorem~\ref{conj}.  In the following, we write $\Omega_{g,n+2}(\sigma):=\Omega_{g,n+2}(e_\sigma)$, i.e. we replace $e_\sigma=e_1^{\otimes n}\otimes e_0^{\otimes m}$ with $\sigma=(1^n,0^m)$, to ease notation in formulae.

\begin{proposition}   \label{OmK}
\[
\int_{\overline{\modm}_{g,n+2}}\hspace{-5mm}\Omega_{g,n+2}(1^n,0^2)\prod_{j=1}^n\psi_j^{k_j}=2\int_{\overline{\modm}_{g,n}}\hspace{-5mm}\bk_{2g-3+n}\prod_{j=1}^n\psi_j^{k_j}+\sum_{j=1}^n\int_{\overline{\modm}_{g,n}}\hspace{-5mm}\bk_{2g-2+n}\psi_j^{k_j-1}\prod_{i\neq j}^n\psi_i^{k_i}.
\]
\end{proposition}
\begin{proof}
Write $\pi=\pi^{(2)}:\overline{\modm}_{g,n+2}\to\overline{\modm}_{g,n}$, where we drop the superscript for ease of notation.  For $i\leq n$,
\[\psi_i=\pi^*\psi_i+D_{i,n+1}+D_{i,n+2}+D_{i,n+1,n+2}.
\]
We will use the following vanishing results:
\[\psi_i\cdot D_{i,j}=0,\quad D_{i,n+1,n+2}\cdot D_{j,n+1,n+2}=0,\quad i\neq j\]
and
\[\Omega_{g,n+2}(1^n,0^2)\cdot\psi_i\cdot D_{i,n+1,n+2}=0
\]
since the restriction to $D_{i,n+1,n+2}$ is given by $\Omega_{0,4}(1^2,0^2)\cdot\psi_i$  which vanishes due to its degree being too large.  Also
\[
\Omega_{g,n+2}(1^n,0^2)\cdot D_{i,n+1}=0 
\]
since  its intersection has the wrong degree:
\[\deg\Omega_{g,n+1}(1^{n-1},0^2)\otimes\Omega_{0,3}(1,0^2)=2g-2+n<2g-2+n+1=\deg\Omega_{g,n+2}(1^n,0^2).\]   
Thus
\begin{align*}
\Omega_{g,n+2}(1^n,0^2)\cdot\psi_j^{k_j}&=\Omega_{g,n+2}(1^n,0^2)\cdot\psi_j\pi^*\psi_j^{k_j-1}\\
&=\Omega_{g,n+2}(1^n,0^2)\cdot(\pi^*\psi_j^{k_j}+D_{j,n+1,n+2}\pi^*\psi_j^{k_j-1})
\end{align*}
and
\begin{align*}
\Omega_{g,n+2}&(1^n,0^2)\cdot\prod_{j=1}^n\psi_j^{k_j}\\
&=\Omega_{g,n+2}(1^n,0^2)\cdot\left(\pi^*\prod_{j=1}^n\psi_j^{k_j}+\sum_{j=1}^nD_{j,n+1,n+2}\pi^*\big\{\psi_j^{k_j-1}\prod_{i\neq j}\psi_i^{k_i}\big\}
\right).
\end{align*}
Therefore
\begin{align*}
\int_{\overline{\modm}_{g,n+2}}\hspace{-3mm}\Omega_{g,n+2}&(1^n,0^2)\prod_{j=1}^n\psi_j^{k_j}\\
&=\int_{\overline{\modm}_{g,n}}\prod_{j=1}^n\psi_j^{k_j}\pi_*\Omega_{g,n+2}(1^n,0^2)+\sum_{j=1}^n\int_{\overline{\modm}_{g,n}}\hspace{-5mm}\bk_{2g-2+n}\psi_j^{k_j-1}\prod_{i\neq j}^n\psi_i^{k_i}\\
&=2\int_{\overline{\modm}_{g,n}}\hspace{-3mm}\bk_{2g-3+n}\prod_{j=1}^n\psi_j^{k_j}+\sum_{j=1}^n\int_{\overline{\modm}_{g,n}}\hspace{-5mm}\bk_{2g-2+n}\psi_j^{k_j-1}\prod_{i\neq j}^n\psi_i^{k_i}
\end{align*}
where the first equality uses 
\[\Omega_{g,n+2}(1^n,0^2)\cdot D_{j,n+1,n+2}=\Omega_{g,n}(1^n)\otimes\Omega_{0,4}(1^2,0^2)\]
together with $\Omega_{g,n}(1^n)=\bk_{2g-2+n}$ and $\int_{\overline{\modm}_{0,4}}\Omega_{0,4}(1^2,0^2)=1$.  The second equality uses $\pi_*\Omega_{g,n+2}(1^n,0^2)=2\bk_{2g-3+n}$.
\end{proof}

\begin{cor}
\[\bar{\cz}(\hbar,s,t_0,t_1,...)=\exp\left\{\tfrac12s^2(\cl_{-1}+t_0)\right\}\cdot\cz^\bk(\hbar,s,t_0,t_1,...)+O(s^4)\]
\end{cor}
\begin{proof}
We can rewrite $1/2!$ times the intersection formulae in Proposition~\ref{OmK} plus
\[
\int_{\overline{\modm}_{g,n}}\hspace{-5mm}\Omega_{g,n}(1^n)\prod_{j=1}^n\psi_j^{k_j}=\int_{\overline{\modm}_{g,n}}\hspace{-5mm}\bk_{2g-2+n}\prod_{j=1}^n\psi_j^{k_j}
\]
in terms of the partition functions:
\begin{align*}
\bar{\cf}(\hbar,s,t_0,t_1,...)&=\cf^\bk(\hbar,s,t_0,t_1,...)+\frac12s^2\sum_{k=0}^\infty t_{k+1}\frac{\partial}{\partial t_k}\cf^\bk(\hbar,s,t_0,t_1,...)\\
&\qquad+\frac12s^2t_0+\frac14s^2t_0^2+O(s^4)
\end{align*}
which is equivalent to the statement of the corollary.
The unstable terms $\frac12s^2t_0$ and $\frac14s^2t_0^2$ do not arise via the relation with intersections of the class $\bk$.  The term $\frac14s^2t_0^2$ is part of $\tfrac12s^2\cl_{-1}$ so we only need to add the extra term $t_0$ to the operator.
\end{proof}
\begin{remark}
The strategy of the proof of Proposition~\ref{OmK} generalises to show that intersection numbers of $\psi$ classes with $\Omega_{g,n+m}(1^n,0^m)$ can be expressed in terms of intersection numbers of $\psi$ classes with $\bk_m$.  In terms of partial differential operators, the relation may be of the form:
\[\bar{\cz}(\hbar,s,t_0,t_1,...)=\exp\left\{\frac{s^2}{2}t_0+\sum_m\frac{s^m}{m!}c_m\cl_{-m}\right\}\cdot\cz^\bk(\hbar,s,t_0,t_1,...)\]
for coefficients $c_m\in\bq$ to be determined.  The appearance of $\cl_{-m}$ is a guess based on the pullback formula for $\psi$ with respect to the forgetful map $\pi^{(2m)}$.  
\end{remark}
\begin{proposition}  \label{virzbar}
\[\bigl((2m+1)!!\frac{\partial}{\partial t_m}-\cl_m-\tfrac{s^2}{2\hbar}\delta_{m,0}\bigr)\bar{\cz}(\hbar,s,\vec{t})=O(s^4).
\]
\end{proposition}
\begin{proof}

Up to $O(s^4)$,
\begin{align*}
\bigl((2m+1)!!\frac{\partial}{\partial t_m}-&\cl_m\bigr)\bar{\cz}=\bigl((2m+1)!!\frac{\partial}{\partial t_m}-\cl_m\bigr)(\cz^\bk+\left\{\tfrac12s^2(\cl_{-1}+t_0)\right\}\cdot \cz^\bk)\\
&=s^2\left\{\cl_{m-1}\cz^\bk+((2m+1)!!\frac{\partial}{\partial t_m}-\cl_m\big)\tfrac12s^2(\cl_{-1}+t_0)\right\}\cdot \cz^\bk\\
\end{align*}
which uses \eqref{virK} to remove the $s^0$ terms.  To calculate the right hand side we need the commutator for $m>0$
\[
\big[(2m+1)!!\frac{\partial}{\partial t_m}-\cl_m,\cl_{-1}+t_0\big]=(2m+1)!!\frac{\partial}{\partial t_{m-1}}-2(m+1)\cl_{m-1}-(2m-1)!!\frac{\partial}{\partial t_{m-1}}.
\]
Thus
\begin{align*}
((2m+1)!!\frac{\partial}{\partial t_m}-\cl_m\big)&(\cl_{-1}+t_0)\cdot \cz^\bk=[((2m+1)!!\frac{\partial}{\partial t_m}-\cl_m\big),\cl_{-1}+t_0]\cdot \cz^\bk\\
&\hspace{2.5cm}+(\cl_{-1}+t_0)\cdot ((2m+1)!!\frac{\partial}{\partial t_m}-\cl_m\big)\cz^\bk\\
&=\Big(2m(2m-1)!!\frac{\partial}{\partial t_{m-1}}-2(m+1)\cl_{m-1}\Big)\cdot \cz^\bk+O(s^2)\\
&=-2\cl_{m-1}\cdot \cz^\bk+O(s^2)
\end{align*}
where both the first and second equalities use \eqref{virK}. Thus we conclude that
\[\bigl((2m+1)!!\frac{\partial}{\partial t_m}-\cl_m\bigr)\bar{\cz}=s^2(\cl_{m-1}\cz^\bk-\cl_{m-1}\cz^\bk)+O(s^4)=O(s^4).\]
\end{proof}
Theorem~\ref{conj} is an immediate consequence of Propositions~\ref{vireq} and \ref{virzbar}.  We see that the proof relies on the intersection formula in Proposition~\ref{OmK} and  the non-trivial algebraic structure given by translation in Theorem~\ref{th:transl} and the Virasoro algebra.  The proof does not give an explanation for why it works.  The following discussion attempts to rectify this via a heuristic differential geometric argument.

%\begin{remark}  
\subsection{Supergeometry} \label{supergeom}
Mirzakhani \cite{MirSim} proved a recursion similar to \eqref{reconj} via an argument which uses the moduli space of hyperbolic surfaces with geodesic boundary.  On any given hyperbolic surface with a distinguished boundary component $\beta$ she produced a one-to-one correspondence between a collection of disjoint intervals in $\beta$ and embedded pairs of pants in the hyperbolic surface that meet %with one boundary component 
$\beta$.  The length of each of the intervals is determined by its corresponding pair of pants.  Mirzakhani analysed an arbitrary hyperbolic pair of pants to produce an expression for the length of such an interval as a function of the lengths $x$, $y$ and $z$ of the three boundary components of the pair of pants, which produces kernels similar to $D(x,y,z)$ and $R(x,y,z)$ defined in the introduction.  The moduli spaces of hyperbolic surfaces with geodesic boundary components of lengths $(L_1,...,L_n)$ have natural symplectic forms, hence also volumes, which correspond to the symplectic deformations $\omega(L_1,...,L_n)$ used in the definition \eqref{nsdef}.  Mirzakhani applied this construction to each point in the moduli space to produce recursive relations between volumes of moduli spaces of hyperbolic surfaces with a given topology and volumes of moduli spaces of hyperbolic surfaces obtained by cutting out a pair of pants from a surface, thus reducing the complexity of the topology of the hyperbolic surface.  The recursive relations involve the kernels and the moduli spaces of simpler complexity as in \eqref{reconj}.

Stanford and Witten \cite{SWiJTG} generalised Mirzakhani's argument to the moduli space of super hyperbolic surfaces with geodesic boundary.  They used this to produce a heuristic proof of \eqref{reconj} in the $s=0$ case, when all boundary components are Neveu-Schwarz.  Via an analysis of super hyperbolic pairs of pants with boundary components of (super) length $x$, $y$ and $z$, they produced the kernels $D(x,y,z)$ and $R(x,y,z)$ defined in the introduction, in the Neveu-Schwarz case. They produced similar kernels when two of the boundary components of the super hyperbolic pair of pants are Ramond.  If we apply their construction to a hyperbolic surface with $n$ marked Neveu-Schwarz geodesic boundary components together with a gas of Ramond punctures we produce a heuristic proof of \eqref{reconj}.  It produces (without proof) a recursion that contains all of \eqref{reconj} plus extra terms with new kernels corresponding to pairs of pants with Ramond boundary components.  The extra terms should give a zero contribution to the recursion, because we expect the volumes
to vanish in the case of Ramond geodesic boundary components.  Stanford and Witten give an argument in \cite{SWiJTG} for the vanishing of these volumes.  Such a proof of the recursion is so far incomplete, nevertheless it is useful to heuristically explain the form of the recursion.  An explanation of the form of the recursion is currently beyond the algebro-geometric methods used in the proofs of Theorems~\ref{main} and \ref{conj}.
%\end{remark}

\subsection{Calculations}
In this section we collect some explicit calculations via intersection theory, and observe the recursion in these cases.  This has been helpful to check formulae for errors.  The reader shlould also find these useful to help to understand the recursion and its proof.

\begin{lemma}   \label{initial}
\[ \widehat{V}^{(2)}_{0,n}=\tfrac{1}{2!}\widehat{V}^{WP}_{0,(1^n,0^2)}=\tfrac12(n-1)!
\]
\end{lemma}
\begin{proof}
When $n=1$, the initial value $\widehat{V}^{WP}_{0,(1,0^2)}=1$ is calculated via the pushforward $\Omega^{\text{spin}}_{0,3}(1,0^2)=p_*c_0=p_*1=\frac12$ and \eqref{Omegadef} to get
\[\Omega_{0,3}(1,0^2)=2^{-1+2}\Omega^{\text{spin}}_{0,3}(1,0^2)=2\cdot\tfrac12=1
\] 
hence $\widehat{V}^{(2)}_{0,n}=\tfrac{1}{2!}\widehat{V}^{WP}_{0,(1,0^2)}=\tfrac{1}{2!}\int_{\overline{\modm}_{0,3}}\Omega_{0,3}(1,0^2)=\tfrac{1}{2}$.

For $n>1$ we have
\begin{align*} 
\widehat{V}^{WP}_{0,(1^n,0^2)}&=\int_{\overline{\modm}_{0,n+2}}\Omega_{0,n+2}(1^n,0^2)=\int_{\overline{\modm}_{0,n+2}}\psi_{n}\pi_n^*\Omega_{0,n+1}(1^{n-1},0^2)\\
&=(n-1)\int_{\overline{\modm}_{0,n+1}}\Omega_{0,n+1}(1^{n-1},0^2)
\end{align*}
where $\pi_n$ forgets the $n$th point, the second equality
uses the pullback formula \eqref{forgbund1}, and the third equality uses the pushforward formula $(\pi_n)_*\psi_n=2g-2+n+1=n-1$.

The lemma follows by induction.
\end{proof}
\begin{lemma}
\[ \widehat{V}_{0,n}^{(4)}(L_1,...,L_n)=\tfrac{1}{4!}(n+2)!\pi^2+\tfrac{1}{4!\cdot4}(n+1)!\sum_{i=1}^n L_i^2
\]
\end{lemma}
\begin{proof}
Since $\deg\Omega_{0,n+4}(1^n,0^4)=n$, we have 
\begin{align*}  
\widehat{V}^{WP}_{0,(1^n,0^4)}&=\int_{\overline{\modm}_{0,n+4}}\hspace{-3mm}\Omega_{0,n+4}(1^n,0^4)\exp(2\pi^2\kappa_1+\frac12\sum L_i^2\psi_i)\\
&=2\int_{\overline{\modm}_{0,n+4}}\hspace{-3mm}\Omega_{0,n+4}(1^n,0^4)\cdot\kappa_1\pi^2
+\frac12\int_{\overline{\modm}_{0,n+4}}\hspace{-3mm}\Omega_{0,n+4}(1^n,0^4)\cdot\psi_i\sum L_i^2
\end{align*}
We have rank $E_{0,(0^4)}$=0 and $c(E_{0,(0^4)})=1$ so $\Omega^{\text{spin}}_{0,4}(0^4)=p_*c(E_{0,(0^4)})=\frac12$ where the pushforward introduces a factor of $\frac12$.  

By \eqref{Omegadef}, $\Omega_{0,4}(0^4)=2\Omega^{\text{spin}}_{0,4}(0^4)=1$, and using the forgetful map $\pi$ which forgets the first point $p_1$, 
$\Omega_{0,5}(1,0^4)=\psi_1\pi^*\Omega_{0,4}(0^4)=\psi_1.$ 
Hence 
\[\int_{\overline{\modm}_{0,5}}\hspace{-2mm}\Omega_{0,5}(1,0^4)\cdot\kappa_1=\int_{\overline{\modm}_{0,5}}\hspace{-1mm}\psi_1\cdot\kappa_1=3,\quad
\int_{\overline{\modm}_{0,5}}\hspace{-2mm}\Omega_{0,5}(1,0^4)\cdot\psi_1=\int_{\overline{\modm}_{0,5}}\hspace{-1mm}\psi_1^2=1
\]
and
\[\widehat{V}^{WP}_{0,(1,0^4)}=\frac{1}{4!}\left(6\pi^2+\tfrac12L_1^2\right)\]
which agrees with the $n=1$ case of the stated formula.

For $n>1$, we use the initial case and induction on the coefficients:
\begin{align*} 
\int_{\overline{\modm}_{0,n+4}}\Omega_{0,n+4}(1^n,0^4)\psi_1&=\int_{\overline{\modm}_{0,n+4}}\psi_{n}\pi_n^*\Omega_{0,n+3}(1^{n-1},0^4)\\
&=(n+1)\int_{\overline{\modm}_{0,n+3}}\Omega_{0,n+3}(1^{n-1},0^4)
\end{align*}
where $(\pi_n)_*\psi_n=2g-2+n+3=n+1$.   By induction, this produces the required coefficient $\tfrac{1}{4!\cdot4}(n+1)!$ of $\sum L_i^2$.
The constant term uses the following.
\begin{align*} 
\int_{\overline{\modm}_{0,n+4}}\Omega_{0,n+4}(1^n,0^4)\kappa_1&=\int_{\overline{\modm}_{0,n+4}}\hspace{-3mm}\psi_{n}\pi_n^*\Omega_{0,n+4}(1^{n-1},0^4)(\pi_n^*\kappa_1+\psi_n)\\
=(n+1)&\int_{\overline{\modm}_{0,n+3}}\hspace{-3mm}\Omega_{0,n+3}(1^{n-1},0^4)\kappa_1+\int_{\overline{\modm}_{0,n+3}}\hspace{-3mm}\Omega_{0,n+3}(1^{n-1},0^4)\kappa_1\\
=(n+2)&\int_{\overline{\modm}_{0,n+3}}\hspace{-3mm}\Omega_{0,n+3}(1^{n-1},0^4)\kappa_1\\
=\tfrac{1}{4!}(n+2&)!\pi^2
\end{align*}
where in the second line the first term uses $(\pi_n)_*\psi_n=n+1$ and the second term uses $(\pi_n)_*\psi_n^2=\kappa_1$.  The final equality uses the initial evaluation and induction.
\end{proof}

%\[ \widehat{V}_{0,n}^{(6)}=\tfrac{1}{2^66!}(n+3)!\left(\sum L_i^4+2\sum L_i^2L_j^2+2^4(n+4)\pi^2\sum L_i^2+2^4(n+4)(2n+9)\pi^4\right)\]
\subsubsection{The recursion}
When $m=1$ we verify
\[
L_1\widehat{V}_{0,n}^{(2)}(L_1,L_K)=
\sum_{j=2}^n\int_0^\infty xR(L_1,L_j,x)\widehat{V}_{0,n-1}^{(2)}(x,L_{K\backslash\{j\}})dx
\]
since $\int_0^\infty xR(L_1,L_j,x)dx=L_1$
and $\widehat{V}_{0,n}^{(2)}=\tfrac12(n-1)!$ is constant.  

When $m=2$ we verify
\begin{align*}  
L_1\widehat{V}_{0,n}^{(4)}(L_1,L_K)&=
\sum_{j=2}^n\int_0^\infty xR(L_1,L_j,x)(\tfrac{1}{96}n!(x^2+\sum_{i\neq1,j} L_i^2)+\tfrac{1}{24}(n+1)!\pi^2)dx\\
&+\frac12\int_0^\infty\int_0^\infty xyD(L_1,x,y)dxdy\underbrace{\sum_{I \sqcup J = K}\tfrac14|I|!\tfrac12|J|!}_{
\frac14n!}\\
&=\tfrac{1}{96}n![(n-1)L_1(L_1^2+12\pi^2)+(n+1)L_1L_j^2]\\
&\hspace{2cm}+\tfrac{1}{24}(n+1)!(n-1)\pi^2L_1
+\tfrac18n!(\tfrac16L_1^3+2\pi^2L_1)\\
&=L_1(\tfrac{1}{96}(n+1)!\sum L_i^2+\tfrac{1}{24}(n+2)!\pi^2)
\end{align*}
which uses
\[ \int_0^\infty\int_0^\infty xyD(L_1,x,y)dxdy=\tfrac16L_1^3+2\pi^2L_1
\]
and
\[\int_0^\infty xR(L_1,L_j,x)dx=L_1,\quad\int_0^\infty x^3R(L_1,L_j,x)dx=L_1(L_1^2+3L_j^2+12\pi^2).
\]
When $m=3$, we can use $\widehat{V}_{0,1}^{(6)}(L_1)$ from the introduction, to check:
\begin{align*}  
L_1\widehat{V}_{0,1}^{(6)}(L_1)&=
\frac12\int_0^\infty\int_0^\infty xyD(L_1,x,y)\widehat{V}_{0,1}^{(2)}(x)\widehat{V}_{0,1}^{(4)}(y)dxdy\times 2\\
&=\frac{1}{96}\int_0^\infty\int_0^\infty xyD(L_1,x,y)(x^2+12\pi^2)dxdy\\
&=\tfrac{1}{96}(\tfrac{1}{20}L_1^5+2\pi^2L_1^3+20\pi^4L_1)+\tfrac18(\tfrac16L_1^3+2\pi^2L_1)\pi^2\\
&=L_1\left(\tfrac{1}{1920}L_1^4+\tfrac{1}{24}\pi^2L_1^2+\tfrac{11}{24}\pi^4\right)
\end{align*}
which uses
\[ \int_0^\infty\int_0^\infty x^3yD(L_1,x,y)dxdy=\tfrac{1}{20}L_1^5+2\pi^2L_1^3+20\pi^4L_1\]
and
\[\int_0^\infty x^5R(L_1,L_j,x)dx=L_1(L_1^4+3L_j^4+12\pi^2).
\]

\end{document}